\documentclass{amsart}
\usepackage{amssymb,
enumitem,
mathrsfs,
mathtools,
tikz,
upgreek,
verbatim,
hyphenat,
url,
amsmath
}
\usepackage[T1]{fontenc}
\usepackage[shadow,loadshadowlibrary]{todonotes}
\usepackage{tikz-cd}

\usetikzlibrary{calc}

\usepackage{etoolbox}
\usepackage{ifthen}
\newcommand{\adddetails}[1]{\ifthenelse{\equal{#1}{0}}{\toggletrue{notes}}{\togglefalse{notes}}}
\newtoggle{notes}
\adddetails{0} 
\newcommand{\Bornot}[2]{\iftoggle{notes}{#1}{#2}}

\usetikzlibrary{shapes.misc, positioning}
\numberwithin{equation}{section}

\usepackage[colorlinks=true,linkcolor={black},citecolor={blue}]{hyperref}%

\newcommand{\markdef}[1]{\textbf{#1}}
\newcommand{\E}{E_{\mathbb{B}}}             
\newcommand{\EH}{E_{\mathbb{H}}}

\newcommand{\Mod}{\operatorname{Mod}}

\newcommand{\embeds}{\sqsubseteq}

\DeclareMathOperator{\rnk}{rnk}

\newenvironment{enumerate-(a)}{\begin{enumerate}[label={\upshape (\alph*)}, leftmargin=2pc]}{\end{enumerate}}
\newenvironment{enumerate-(a)-r}{\begin{enumerate}[label={\upshape (\alph*)}, leftmargin=2pc,resume]}{\end{enumerate}}
\newenvironment{enumerate-(a)-5}{\begin{enumerate}[label={\upshape (\alph*)}, leftmargin=2pc,start=5]}{\end{enumerate}}
\newenvironment{enumerate-(A)}{\begin{enumerate}[label={\upshape (\Alph*)}, leftmargin=2pc]}{\end{enumerate}}
\newenvironment{enumerate-(A)-r}{\begin{enumerate}[label={\upshape (\Alph*)}, leftmargin=2pc,resume]}{\end{enumerate}}
\newenvironment{enumerate-(i)}{\begin{enumerate}[label={\upshape (\roman*)}, leftmargin=2pc]}{\end{enumerate}}
\newenvironment{enumerate-(i)-r}{\begin{enumerate}[label={\upshape (\roman*)}, leftmargin=2pc,resume]}{\end{enumerate}}
\newenvironment{enumerate-(I)}{\begin{enumerate}[label={\upshape (\Roman*)}, leftmargin=2pc]}{\end{enumerate}}
\newenvironment{enumerate-(I)-r}{\begin{enumerate}[label={\upshape (\Roman*)}, leftmargin=2pc,resume]}{\end{enumerate}}
\newenvironment{enumerate-(1)}{\begin{enumerate}[label={\upshape (\arabic*)}, leftmargin=2pc]}{\end{enumerate}}
\newenvironment{enumerate-(1)-r}{\begin{enumerate}[label={\upshape (\arabic*)}, leftmargin=2pc,resume]}{\end{enumerate}}
\newenvironment{itemizenew}{\begin{itemize}[leftmargin=2pc]}{\end{itemize}}

\newtheorem{theorem}{Theorem}[section]
\newtheorem{lemma}[theorem]{Lemma}
\newtheorem{corollary}[theorem]{Corollary}
\newtheorem{proposition}[theorem]{Proposition}
\newtheorem{question}[theorem]{Question}
\newtheorem{questions}[theorem]{Questions}
\newtheorem{problem}[theorem]{Problem}
\newtheorem{fact}[theorem]{Fact}
\newtheorem{claim}{Claim}[theorem]
\theoremstyle{definition}
\newtheorem{definition}[theorem]{Definition}

\newtheorem{example}[theorem]{Example}
\newtheorem{conjecture}[theorem]{Conjecture}
\theoremstyle{remark}
\newtheorem{remark}[theorem]{Remark}

\newtheorem{innercustomgeneric}{\customgenericname}
\providecommand{\customgenericname}{}
\newcommand{\newcustomtheorem}[2]{%
  \newenvironment{#1}[1]
  {%
   \renewcommand\customgenericname{#2}%
   \renewcommand\theinnercustomgeneric{##1}%
   \innercustomgeneric
  }
  {\endinnercustomgeneric}
}

\newcustomtheorem{customthm}{Theorem}
\newcustomtheorem{customlemma}{Lemma}

\begin{document}

\title[Structural results on idealistic equivalence relations]{Structural results on \\ idealistic equivalence relations}

\date{}
\author[F.~Calderoni]{Filippo Calderoni }
\address{Department of Mathematics, Rutgers University, 
Hill Center for the Mathematical Sciences,
110 Frelinghuysen Rd.,
Piscataway, NJ 08854-8019}
\email{filippo.calderoni@rutgers.edu}

\author[L.~Motto~Ros]{Luca Motto~Ros }
\address{Dipartimento di matematica \guillemotleft Giuseppe Peano\guillemotright, Via Carlo Alberto 10, Universit\`a di Torino, 10123 Torino, Italy}
\email{luca.mottoros@unito.it}

\begin{abstract}
We address some fundamental problems concerning the structure of idealistic equivalence relations. 
In particular, we show that, under analytic determinacy, there are continuum many idealistic analytic equivalence relations that are not classwise Borel isomorphic to an orbit equivalence relation.
Moreover, we also discuss alternative formulations of the \(E_1\) conjecture, and present an elementary counterexample to one of its earliest variants proposed by Hjorth and Kechris in 1997.
\end{abstract}

\thanks{The \Bornot{first}{second} author is supported by the NSF Grant DMS -- 2348819. The \Bornot{second}{third} author is supported by the Italian PRIN 2022 ``Models, sets and classifications'', prot.\ 2022TECZJA, and is a member of INdAM-GNSAGA. We thank A.\ Kechris, A.\ Shani, and S.\ Solecki for useful comments. The results of Section~\ref{sec:Becker} are based on unpublished work of H.\ Becker~\cite{Bec}: we are grateful to him for sharing his unpublished notes and for allowing us to include some of his arguments.}

\maketitle



\section{Introduction}

\subsection{Background} \label{subsec:introbackground}

Understanding the quotient of Polish or standard Borel spaces modulo definable equivalence relations is a major program in modern descriptive set theory. Most of the effort concentrates on the class of analytic equivalence relations, which are compared to each other using the notion of Borel reducibility and organized in complexity degrees. 
Given two equivalence relations \( E \) and \( F \) on standard Borel spaces \( X \) and \( Y \), we say that \( E \) is \textbf{Borel reducible} to \( F \), in symbols \( E \leq_B F \), if there is a Borel function \( f \colon X \to Y \) such that \( x \mathrel{E} x' \iff f(x) \mathrel{F} f(x') \), for all \( x,x' \in X \). 
We denote by \( \sim_B \) the equivalence relation induced by the preorder \( \leq_B \), which is called \textbf{Borel bireducibility}. 
The task of understanding the structure induced by \( \leq_B \) 
has been a challenge for over the last four decades.

A crucial role in the theory of analytic equivalence relations is played by the so-called \textbf{orbit equivalence relations}, namely, by those equivalence relations whose classes are the orbits induced by a Borel action of a Polish group on some standard Borel space. A related notion is the following. Given an analytic equivalence relation \(E \) on a standard Borel space \( X \), we say that \( E \) is \textbf{idealistic} if there is a map assigning to each equivalence class \( C \in X/E \) a nontrivial ccc \(\sigma\)-ideal \( I_C \) on \( C \) such that \( C \notin I_C \) and the map \( C \mapsto I_C \) is Borel, in the sense that for each Borel \( A \subseteq X^2 \), the set \( A_I  \subseteq X \) defined by 
\begin{equation} \label{eq:idealistic}
x\in A_I \iff \{ y \in [x]_E \mid (x,y) \in A \} \in I_{[x]_E} 
\end{equation}
is Borel.
The terminology ``idealistic'' first appeared in the paper~\cite{KecLou97} by Kechris and Louveau, but one can trace the notion back to Kechris' work about orbit equivalence relations induced by locally compact group actions (see~\cite{Kec92,Kec94}).

A lot of speculation has surrounded the nuance  between idealistic and orbit equivalence relations
for over thirty years.
While orbit equivalence relations have been widely studied in the literature, very little is known about idealistic equivalence relations and their structure under \( \leq_B \).
As observed in~\cite[Section~1.II]{Kec92}
all orbit equivalence relations are idealistic, so it is natural to raise the following question.

\begin{question}[Folklore] \label{question:basic}
Is it true that every idealistic equivalence relation on a standard Borel space is Borel bireducible with an orbit equivalence relation?
\end{question}

Although fundamental, the above question is still open. It admits several variants. For example, one could restrict the attention to specific equivalence relations, like the Borel ones or the analytic equivalence relations having only Borel classes. (The latter encompasses all orbit equivalence relations, including those which are proper analytic.) Moreover one could consider natural variants of Borel bireducibility that have appeared in the literature, such as Borel isomorphism or classwise Borel isomorphism (\cite[Definition~2.1]{Mot12}; see also Definition~\ref{def:classwise} for the latter).

A fruitful approach to uncover structural results for definable equivalence relations consists of proving dichotomy theorems. 
The prototype of a dichotomy theorem would state that for every equivalence relation in a given class, either it satisfies some regularity property or, otherwise, a certain object is Borel reducible to it; such an object would thus provide a canonical obstruction against the regularity property from the first alternative.

In (\cite[Theorem 4.1]{KecLou97}) it is proved that an obstruction to a Borel \( E \) being 
idealistic is that \( E_1 \leq_B E \), where \( E_1 \) is the relation of eventual agreement on countable sequences of reals.
This led Kechris and Louveau to ask if this can be turned into a dichotomy  (see~\cite[p.\ 241]{KecLou97}), which would thus become a very elegant characterization of the class of idealistic Borel equivalence relations.

\begin{problem}[Kechris-Louveau] \label{conj1intro}
If \( E \) is a Borel equivalence relation, is it true that either \( E_1 \leq_B E \) or \( E \)
is idealistic?
\end{problem}

Solecki~\cite{Sol09} gave a positive answer to Problem~\ref{conj1intro} in the restricted context of coset equivalence relations induced by the translation action of Abelian Borel subgroups of Polish groups.
Nevertheless, the dichotomy from Problem~\ref{conj1intro} is in general too strong, and admits easy counterexamples.%
\footnote{We thank A.\ Shani for pointing it out to us.}
Indeed, in~\cite[Theorem~2.4]{Kec92} it is proven that if \( E \) is an equivalence relation on a standard Borel space \( X \), then \( E \) has
a Borel selector if and only if \( E \) is smooth and idealistic. Therefore, every smooth equivalence relation \( E \) without a Borel selector is not idealistic, yet \( E_1 \nleq_B E \). However, there are variants of Problem~\ref{conj1intro} which are not ruled out by the previous counterexamples. For example, Hjorth and Kechris proposed the following conjecture, which can be found in~\cite[Conjecture 1]{HjoKec} and~\cite[Conjecture 6.1]{HjoKec2}:

\begin{conjecture}[Hjorth-Kechris] \label{conj2intro}
For \( E \) a Borel equivalence relation, either \( E_1 \leq_B E \) or \( E \) is Borel bireducible with an orbit equivalence relation.
\end{conjecture}

Conjecture~\ref{conj2intro} can be further weakened by requiring that if \( E_1 \nleq_B E\), then \( E \) is Borel bireducible with an \emph{idealistic} equivalente relation. These two versions of the conjecture are obviously related to Question~\ref{question:basic}, and could even be unified if the latter admits a positive answer.

\subsection{Content of the paper} \label{subsec:intro2}

The main result of the paper is a strong negative answer to one of the possible variations of Question~\ref{question:basic}. Let \( \mathcal{I} \) be the collection of all idealistic analytic equivalence relations which have only Borel equivalence classes 
and are not \emph{classwise Borel isomorphic} to an orbit equivalence relation.
(For the precise definition of class-wise Borel isomorphism see Definition~\ref{def:classwise}.)

\begin{theorem} \label{thm:intro1}
Assume \( \boldsymbol{\Sigma}^1_1 \)-determinacy. Then there is an embedding of \( (\mathcal{O}, {\leq_B}) \) into \( (\mathcal{I},{\leq_B}) \), where \( \mathcal{O} \) is the class of all Borel orbit equivalence relations with uncountably many orbits.

In particular, there are \( 2^{\aleph_0} \)-many \( \leq_B \)-incomparable idealistic equivalence relations which are not classwise Borel isomorphic to an orbit equivalence relation.
\end{theorem}

To fully appreciate Theorem~\ref{thm:intro1}, recall that the structure of \( (\mathcal{O}, {\leq_B}) \) is extremely rich. For example, 
Adams and Kechris~\cite{AdaKec} proved that the partial order of inclusion modulo finite sets on \(\mathcal{P}(\omega)\) embeds into \( (\mathcal{O}, {\leq_B}) \).

It must be noted that Theorem~\ref{thm:intro1} does not fully answer Question~\ref{question:basic}:
this is because classwise Borel isomorphism does not coincide with Borel bireducibility on non-Borel (even orbit) analytic equivalence relations --- see Section~\ref{sec:idealistic} for more details. Moreover, it would be desirable to remove the determinacy assumption from Theorem~\ref{thm:intro1}, and get a purely \( \mathsf{ZFC} \)-result.

We also observe that Question~\ref{question:basic} has instead a positive answer if we restrict to idealistic equivalence relations which are \( \leq_B\)-below some Borel orbit equivalence relation (Proposition~\ref{prop:intro2}), or \( \leq_B \)-below some orbit equivalence relation induced by a Borel action of the infinite symmetric group \( S_\infty \) (Corollary~\ref{cor:intro2}).

\begin{proposition}\label{prop:intro2} 
Let \( E \) be an idealistic (Borel) equivalence relation, and suppose that \( E \leq_B F \) for some Borel orbit equivalence relation \( F \). Then \( E \) is classwise Borel isomorphic to (and hence Borel bireducible with) a Borel orbit equivalence relation.
\end{proposition}

H.\ Friedman proved in~\cite[Theorem 1.5]{HFri} that if \( E \) is a Borel equivalence relation which is classifiable by countable structures (equivalently: \( E \) is Borel reducible to a Borel action of the infinite symmetric group \( S_\infty \)), then \( E \) is also Borel reducible to a \emph{Borel} orbit equivalence relation \( F \) --- in fact, \( F \) can be presented as an isomorphism relation among certain countable structures.  
Combining this with Proposition~\ref{prop:intro2} we immediately obtain:

\begin{corollary} \label{cor:intro2}
Let \( E \) be an idealistic Borel equivalence relation. If \( E \) is classifiable by countable structures, then \( E \) is classwise Borel isomorphic to (and hence Borel bireducible with) a Borel orbit equivalence relation.
\end{corollary}

Furthermore, we show that an apparently unnoticed combination of known results 
disproves both Conjecture~\ref{conj2intro} and its weakening mentioned at the very end of Section~\ref{subsec:introbackground}.

\begin{proposition} \label{thm:intro3}
There is a Borel equivalence relation \( E \) such that \( E_1 \nleq_B E \), yet \( E \) is not Borel bireducible with an idealistic
equivalence relation.
\end{proposition}

This easy observation prompts us to revise the lists of problems and conjectures originally proposed by Hjorth, Kechris, and Louveau over the years (see Section~\ref{sec:questions}, and in particular Conjecture~\ref{conj3intro}). 

\subsection{Organization of the paper} 

In Section~\ref{sec:idealistic} we revise some known results and definitions, and show how to combine them to obtain some easy yet novel structural results concerning idealistic equivalence relations. In particular, we provide a new sufficient condition for being idealistic (Proposition~\ref{prop:select}); we prove Proposition~\ref{prop:intro2} (see Corollary~\ref{cor:Borelidealistic}); we prove Proposition~\ref{thm:intro3}, thus disproving Conjecture~\ref{conj2intro}, using a certain Borel equivalence relation \( \EH \) isolated by Hjorth (Proposition~\ref{prop:Hjoessctbl}); we observe that \( \EH \) also shows that many interesting classes of analytic equivalence relations are not downward closed under Borel reducibility (Corollary~\ref{cor:hjorth}). 

Section~\ref{sec:Beckertutto} is instead devoted to the proof of the main result of the paper, namely Theorem~\ref{thm:intro1}. This is obtained by applying a very natural construction (Theorem~\ref{thm : main2}) to a specific example, denoted by \( \E \) in this paper, of a member of \( \mathcal{I} \). The definition of \( \E \), which is due to \Bornot{Becker}{the first author}, is strictly related to the isomorphism relation on \( p \)-groups (Sections~\ref{subsec:p-groups} and~\ref{subsec:maindefinition}), and
the proof that, under \(\boldsymbol{\Sigma}^1_1\)-determinacy, \( \E \in \mathcal{I} \) (Section~\ref{sec:Becker}) builds on the unpublished notes~\cite{Bec}. Quite interestingly, such proof involves Steel's technique of forcing with tagged trees. Since the reader may not be familiar with this method, we include an almost
self-contained presentation of Steel's forcing in Appendix~\ref{appendix}.

Finally, Section~\ref{sec:questions} discusses some fundamental questions and open problems concerning the structure of idealistic equivalence relations.

\section{Idealistic equivalence relations} \label{sec:idealistic}

Recall from Section~\ref{subsec:introbackground} the definition of an idealistic equivalence relation. Here are some examples from ~\cite[Section~1.II]{Kec92}.

\begin{example} \label{xmp:idealistic}
\begin{enumerate-(a)}
\item \label{xmp:idealistic-a}
All orbit equivalence relations are idealistic. In fact, if
\(E_a\) is induced by the Borel action \(a\colon G\times X \to X\) we can assign a ccc \(\sigma\)-ideal \(I_C\) to each \(C =[x]_{E_a}\in X/{E_a}\) by
\[
A \in I_C\quad \iff \quad \text{\(\{
g\in G : a(g,x) \in A
\}\) is meager}.
\]
The assignment does not depend on the choice of \(x \in C\).
 By the Feldman-Moore theorem, this includes all countable Borel equivalence relations.
\item \label{xmp:idealistic-b}
Let \(E\) be a Borel equivalence relation on a standard Borel space \(X\), and let \( x\mapsto \mu_x \) be a Borel map from \(X\) into the standard Borel space \( P(X) \) of all probability measures on \(X\) such that \(\mu_x([x]_E)=1\) and \(x\mathrel{E} y \Rightarrow \mu_x\sim \mu_y\). Then \(E\) is idealistic.
\end{enumerate-(a)}    
\end{example}

Of course not all equivalence relations are idealistic --- there are (Borel) equivalence relations which are not even Borel reducible to an idealistic equivalence relation. Indeed, a typical obstruction to \(E\) being idealistic is the following (\cite[Theorem 4.1]{KecLou97}), where \(E_1\) denotes the Borel equivalence relation on \(\mathbb{R}^\mathbb{N}\) defined by 
\[
(x_n)_{n\in \mathbb{N}} \mathrel{E_1} (y_n)_{n\in\mathbb{N}} \iff \exists n \forall m \geq n \, (x_m = y_m).
\] 

\begin{theorem}[Kechris-Louveau]
\label{thm : idealistic obstruction}
If \(E\) is an idealistic Borel equivalence relation, then \(E_1 \nleq_B E\).
\end{theorem} 

Thus \( E_1 \nleq_B E \) is a necessary condition for \( E \) being idealistic. 
The following fact, abstracted from an argument contained in~\cite{Bec}, provides instead a sufficient condition, and will be used in Section~\ref{sec:Beckertutto} to obtain new examples of idealistic analytic equivalence relations.

\begin{definition} \label{def:select}
Suppose that \( F' \subseteq F \) are analytic equivalence relations on the same set \(Z\). If \(\theta\colon Z \to Z\) is a homomorphism from \(F\) to \(F'\) such that \({\theta(z)}\mathrel{F}z\)
for all \(z \in Z\), then we say that \(\theta\) \markdef{selects an \(F'\)-class within each \(F\)-class}. (Of
course, this implies that \(\theta\) is a reduction from \(F\) to \(F'\).)
\end{definition}

\begin{proposition} \label{prop:select} 
Let \( E \) be an orbit equivalence relation induced by a Borel action \( a \colon G \times Z \to Z \) of a Polish group \( G \) on a standard Borel space \( Z \), and let \( F \supseteq E \) be any equivalence relation on \( Z \). If there is a Borel map \(\theta \colon Z\to Z\) selecting an \(E \)-class within every \( F \)-class, then \( F \) is idealistic. 
\end{proposition}

\begin{proof}
Let \(B\mapsto I_B^*\) be the Borel assignment taking any \(G\)-orbit \(B\) to the ccc ideal on \(B\) induced by the meager ideal on \(G\). (For more details see Example~\ref{xmp:idealistic}\ref{xmp:idealistic-a}.)
For any \(F\)-equivalence class \( C \), let
\[
I_C = \left\{D\subseteq C \mid D \cap [\theta(C)]_E \in I^*_{[\theta(C)]_E}\right\}.
\]
We claim that the map \( C \mapsto I_C \) witnesses that \( F \) is idealistic.
First note that, for any \(C\in Z/F\), the corresponding ideal \(I_C\) is ccc. In fact, every antichain \(\{D_j: j\in J\}\) of pairwise disjoint (nonempty) subsets of \(C\) with \(D_j\notin I_C\) induces the family \(\{D_j\cap [\theta(C)]_E: j\in J\}\) of pairwise disjoint subsets of \([\theta(C)]_E\) that are not members of \(I^{*}_{[\theta(C)]_E}\).
It follows that each set \(D_j\cap [\theta(C)]_E\) is necessarily nonempty. Therefore, \(J\) must be a countable set because the ideal \(I^*_{[\theta(C)]_{E}}\) is ccc.

Moreover, for any Borel \(A\subseteq Z\times Z\), the set \(A_I\) from~\eqref{eq:idealistic} is Borel because for every \( z \in Z \)
\begin{align*}
z\in A_I\quad \iff & \quad \{ y\in [z]_F \mid (z,y)\in A \} \in I_{[z]_F}\\
\iff & \quad \{ y\in [\theta(z)]_{E} \mid (z,y)\in A \} \in I^*_{[\theta(z)]_{E}}\\
\iff &\quad (z, a(g,\theta(z)))\notin A \text{ for comeagerly many \(g\in G\)}. \qedhere
\end{align*}
\end{proof}

Another ingredient which will be crucial in our analysis is given by the notions of classwise Borel isomorphism and classwise Borel embeddability, that were implicitly introduced in~\cite{FriMot} and made explicit in~\cite[Definition 2.1 and 2.2]{Mot12}. 

\begin{definition} \label{def:classwise}
Given two analytic equivalence relations \(E, F\) on standard Borel
spaces \(X, Y\), respectively. We say that \(E \) is \markdef{classwise Borel isomorphic} to \(F\), in symbols \(E\simeq_{cB}F\), if 
if there are Borel reductions \(f\colon X \to Y\) and \(g\colon Y \to X\)
of \(E\) into \(F\) and \(F\) into \(E\), respectively, such that their factorings to the quotient
spaces \(\hat f\colon X/E \to Y /F\) and \(\hat g \colon Y /F \to X/E\) are bijections and \(\hat f^{-1} = \hat g\). In this case, we say that \(g\) is a \markdef{classwise Borel inverse} of \(f\).

Moreover, we say that \(E\) is \markdef{classwise Borel embeddable} into \(F\), in symbols \(E \sqsubseteq_{cB} F\), if there is a Borel \(F\)-invariant subset \(A \subseteq Y\) such that
\(E \simeq_{cB} F \restriction A\).
\end{definition}

Classwise Borel embeddability lies in between the classical notions of Borel reducibility \( \leq_B \) and Borel invariant embeddability \( \sqsubseteq^i \) (i.e.\ injective Borel reducibility with invariant range) appeared as early as in \cite{DouJacKec}.
It is clear that \( {E \sqsubseteq^i F} \Rightarrow {E \sqsubseteq_{cB} F} \) and \( {E \sqsubseteq_{cB} F} \Rightarrow {E \leq_B F} \).
However, such reducibility notions are different from each other. For example, while graph isomorphisms Borel reduces to isomorphism on countable linear orders (\cite[Theorem 3]{FriSta}), it does not classwise Borel embed into it (see~\cite[Theorem 4]{Gao01}). Moreover, in~\cite[Theorem 2.10]{Mot12} it is shown that for every orbit equivalence relation \( E \) with perfectly many classes, including the Borel ones, there is \( F \sim_B E\) such that \( F \not\sqsubseteq_{cB} E \); by a result of Kechris and Macdonald such an \( F \) cannot be idealistic when \( E \) is Borel (see Proposition~\ref{prop:kechrismacdonald}).
On the other hand, Borel invariant embeddability is strictly finer then classwise Borel embeddability: this is because  \( \sqsubseteq^i\) also imposes cardinality restrictions on the size of the equivalence classes, while  \( \sqsubseteq_{cB} \) does not. Similar considerations can be made about Borel bireducibility, classwise Borel isomorphism, and Borel isomorphism.

Now we discuss a few structural results concerning \( \simeq_{cB} \) and \( \sqsubseteq_{cB} \) that are relevant to this paper. First of all, in contrast to what happens with classical Borel reducibility, we have a Cantor-Schr\"oder-Bernstein theorem (\cite[Proposition 2.3]{Mot12}):

\begin{fact} \label{fact:CantorShroederBernstein}
For all analytic equivalence relations \( E \) and \( F \),
\[
E \simeq_{cB} F \iff E \sqsubseteq_{cB} F \wedge F \sqsubseteq_{cB} E.
\]
\end{fact}

Next, we have the following important result (\cite[Lemma 3.8]{KecMac}):

\begin{proposition}[Kechris-Macdonald] \label{prop:kechrismacdonald}
If \( E \) is an idealistic equivalence relation and \( F \) is a Borel equivalence relation, then
\[  
E \leq_B F \iff E \sqsubseteq_{cB} F.
\]
\end{proposition}

Proposition~\ref{prop:intro2} follows from part~\ref{cor:Borelidealistic-2} of the next corollary.

\begin{corollary} \label{cor:Borelidealistic}
Let \( E \) and \( F \) be idealistic equivalence relations.
\begin{enumerate-(1)}
\item \label{cor:Borelidealistic-1}
If both \( E \) and \( F \) are Borel, then
\[
E \leq_B F \iff E \sqsubseteq_{cB} F \qquad \text{and} \qquad E \sim_B F \iff E \simeq_{cB} F.
\]
\item \label{cor:Borelidealistic-2}
If \( F \) is Borel orbit equivalence relation and \( E \leq_B F \), then \( E \simeq_{cB} F' \) for some Borel orbit equivalence relation \( F' \).
\end{enumerate-(1)}
In particular, ``idealistic'' coincides with ``orbit'', up to classwise Borel isomorphism, within the class of equivalence relations that are Borel reducible to a Borel orbit equivalence relation.
\end{corollary}

\begin{proof}
Part~\ref{cor:Borelidealistic-1} follows from Proposition~\ref{prop:kechrismacdonald} and Fact~\ref{fact:CantorShroederBernstein}. 
As for part~\ref{cor:Borelidealistic-2}, first use part~\ref{cor:Borelidealistic-1} to obtain \( E \sqsubseteq_{cB} F \), as witnessed by the Borel \( F \)-invariant set \( A \): letting \( F' \) be the restriction of \( F \) to such \( A \), we are done.
\end{proof}

Along the same lines, we borrow from~\cite[Proposition~2.9]{Mot12} the following result:%
\footnote{Here and throughout the paper, \( E_1 \oplus E_2 \) denotes the disjoint union of the equivalence relations \( E_1 \) and \( E_2 \).}

\begin{proposition} \label{prop:absoption} 
For every orbit equivalence relation \( F \), the following are equivalent:
\begin{enumerate-(i)}
\item 
\(\mathrm{id}_\mathbb{R} \leq_B F \);
\item 
\( \mathrm{id}_{\mathbb{R}} \sqsubseteq_{cB} F \);
\item 
\(F \sim_B \mathrm{id}_\mathbb{R} \oplus F\).
\end{enumerate-(i)}
\end{proposition}

In order to prove Proposition~\ref{thm:intro3}, we further need the following result of Hjorth (\cite[Theorem~0.1]{Hjo05}).

\begin{proposition}[Hjorth] \label{prop:Hjoessctbl}
There is a Borel equivalence relation \( \EH \) such that \(\EH \leq_B F\) for some countable Borel equivalence relation \(F\), yet \(\EH\) is not Borel bireducible with any countable Borel equivalence relation.
\end{proposition}

\begin{proof}[Proof of Proposition~\ref{thm:intro3}]
Let \( \EH \) be as in Proposition~\ref{prop:Hjoessctbl}, and let \( F \) be a countable Borel equivalence relation such that \( \EH \leq_B F \). In particular, \( F \) is idealistic by Example~\ref{xmp:idealistic}\ref{xmp:idealistic-a}.
Then \( E_1 \not\leq_B \EH \) because otherwise \( E_1 \leq_B F \), contradicting Theorem~\ref{thm : idealistic obstruction}.
Moreover, if \( \EH \sim_B E' \) for some idealistic equivalence relation \( E' \), then \( E' \leq_B F \), and so \( E' \sqsubseteq_{cB} F \) by Proposition~\ref{prop:kechrismacdonald}: let \( A \) be a Borel \( F \)-invariant set witnessing this. Then \( E' \simeq_{cB} F' \) for \( F'\) the restriction of \( F \) to \( A \), and hence \( \EH \sim_B F' \). But since \( F'\) is a countable Borel equivalence relation too, this contradicts the statement of Proposition~\ref{prop:Hjoessctbl}.
\end{proof}

\begin{remark}
Some authors drop the ccc condition from the definition of idealistic. The equivalence relation \( \EH \) would still be a counterexample to Conjecture~\ref{conj2intro} even if such alternative definition is adopted, as Proposition~\ref{prop:kechrismacdonald} hold in that broader context.
\end{remark}

We also point out the following easy consequence of Proposition~\ref{prop:Hjoessctbl}.

\begin{corollary} \label{cor:hjorth}
Let \( \mathcal{E}\) be a collection of analytic equivalence relations such that every \( E \in \mathcal{E} \) is idealistic, and \( \mathcal{E} \) contains all countable Borel equivalence relations.
Then \( \mathcal{E} \) is not downward closed under Borel reducibility.
\end{corollary} 

\begin{proof}
Let \( \EH \) be as in Proposition~\ref{prop:Hjoessctbl}. If \( \mathcal{E} \) were downward closed under \( \leq_B \), then \( \EH \in \mathcal{E} \), and so \( \EH \) would be idealistic. But as shown in the proof of Proposition~\ref{thm:intro3}, this is impossible.
\end{proof}

In particular, Corollary~\ref{cor:hjorth} applies to the collections of all idealistic analytic equivalence relation and all orbit equivalence relations, as well as to their restrictions to Borel equivalence relations.

\section{Becker's equivalence relation \( \E \)} \label{sec:Beckertutto}

Recall from Section~\ref{subsec:intro2} that \( \mathcal{I} \) denotes the collection of all idealistic analytic equivalence relations such that all their equivalence classes are Borel and they are not classwise Borel isomorphic to any orbit equivalence relation.
Notice that the latter is equivalent to the seemingly stronger requirement of not being classwise Borel embeddable into an orbit equivalence relation.

The core result of this section is the following:

\begin{theorem} \label{thm : main}
Assume \(\boldsymbol{\Sigma}^1_1\)-determinacy. Then \( \mathcal{I} \neq \emptyset \).
\end{theorem}

Indeed, \Bornot{Becker}{the first author} provided in~\cite{Bec} an example \( \E \) of an idealistic analytic equivalence relation with only Borel classes which is not an orbit equivalence relation itself (see Definition~\ref{def:EdiBecker}).
Therefore, to prove Theorem~\ref{thm : main} we just need to show that \( \E \) does not classwise Borel embed into any orbit equivalence relation. This is postponed to Section~\ref{sec:Becker} (see Lemmas~\ref{lem:steel} and~\ref{lemma:contradiction1}).

Denote by \( \mathcal{O} \) the class of all Borel orbit equivalence relations with uncountably many orbits.

\begin{theorem} \label{thm : main2}
If \( E \in \mathcal{I} \) is such that \( E_0 \not\leq_B  E \)
and \(F_1,F_2 \in \mathcal{O}\), then \(E \oplus F_i \in \mathcal{I}\) and
\[
F_1\leq_B F_2 \iff E\oplus F_1\leq_B E \oplus F_2.
\]
\end{theorem}

\begin{proof}
The forward direction of the equivalence is obvious, so we only prove the backwards implication.
Let \(f\colon X\sqcup X_1 \to X\sqcup X_2\) witness \(E\oplus F_1 \leq_B E\oplus F_2\), where \( X \) is the domain of \( E \) and \( X_1 \) and \( X_2 \) are the domains of \( F_1 \) and \( F_2 \), respectively. Let \(F \) be the restriction of \( F_1 \) to the Borel \( F_1 \)-invariant set \( Y_1 = f^{-1}(X)\cap X_1 \), so that \( F \leq_B E \), as witnessed by \( f \restriction Y_1 \).
Note that \( F \) is Borel and \(E_0 \nleq_B F\), because otherwise \(E_0 \leq_B E\), contradicting our assumption.
Therefore \( F \leq_B \mathrm{id}_\mathbb{R}\) by the Borel Glimm-Effros dichotomy. (E.g., see \cite[Theorem~3.4.3]{BecKec}.)
Let \( R  \) be the restriction of \( F_1 \) to \( X_1 \setminus Y_1 \), so that \(F_1 = F \oplus R\) and \( R \leq_B F_2 \), as witnessed by \( f \restriction (X_1 \setminus Y_1) \). If \(R\) has countably many classes, then \( F_1 = F \oplus R \leq_B \mathrm{id}_\mathbb{R} \), and so, as desired, we conclude that \(F_1 \leq F_2\) since \(\mathrm{id}_\mathbb{R} \leq_B F_2\) by Silver's dichotomy. (E.g., see~\cite[Theorem~3.3.1]{BecKec}.) If instead \( R \) has uncountably many classes, and hence \( R \in \mathcal{O} \), then \( \mathrm{id}_{\mathbb{R}}\leq_B R\) by Silver's dichotomy again, and thus by Proposition~\ref{prop:absoption} we get
\[
F_1 = F\oplus R \leq_B \mathrm{id}_\mathbb{R} \oplus R \sim_B R \leq_B F_2,
\]
and we are done.
\end{proof}

Together with (the proof of) Theorem~\ref{thm : main}, Theorem~\ref{thm : main2} entails
Theorem~\ref{thm:intro1}: 
this is because, by construction, Becker's equivalence relation \( \E \) is Ulm classifiable, and thus
\( E_0 \not\leq_B \E \) (see~\cite[Theorem~3.4.4]{BecKec}).

\subsection{Descriptive and computational aspects of \(p\)-groups} \label{subsec:p-groups}

Throughout this section, let \(H\) be a countable abelian \(p\)-group, for a fixed prime number \(p\). (Recall that an abelian group \(H\) is a \textbf{\(p\)-group} if for every \(a \in H\) there exists \(n \in \omega\) such that \(p^n a=0\).)  We say that \(H\) is \textbf{divisible} if for all \(a \in H\) and \(n > 0\) there exists
\(b \in H\) such that \(nb = a\), and that \(H\) is \textbf{reduced} if \(H\) has no nontrivial divisible subgroup. It is well-known that
every countable abelian \(p\)-group \(H\) can be written as a direct sum \(R(H)\oplus D(H)\), where \(R(H)\) is reduced and \(D(H)\) is divisible. In particular, \(D(H)\cong \mathbb{Z}(p^\infty)^{(n)}\) is the direct sum of \(n\) copies of the quasi-cyclic \(p\)-group \(\mathbb{Z}(p^\infty)\) for some \(n\in \omega+1\); we refer to such \(n\) as the \textbf{rank of the divisible part} of \(H\).  Define a transfinite sequence of subgroups \( p^\alpha H \) of \( H \), where \( \alpha \) is an ordinal, by setting
\( p^0 H = H\), \(p^{\alpha+1}H = p(p^\alpha H)\), and for \(\lambda\) limit, \(p^\lambda H =
\bigcap_{\alpha<\lambda} p^\alpha H \).
Whenever \(H\) is countable,
there is a least countable ordinal \(\rho\) such that 
\(p^\rho H  = p^{\rho+1} H \) and thus \(p^\rho H = D(H)\). 
Such ordinal \(\rho\) is called the \markdef{Ulm length} of \(H\) and will be denoted by \(l(H)\).
For all \(\alpha<l(H)\) the \markdef{\(\alpha\)-th Ulm invariant \( U_\alpha(H) \in \omega \cup \{ \infty \} \) of \(H\) }is defined by
\[
U_\alpha(H) =\dim_{\mathbb{Z}_p}(P_\alpha/P_{\alpha+1}),
\] 
where \(P_\alpha = \{a \in p^\alpha H : pa = 0\}\). 
For notational simplicity, when speaking of Ulm invariants we write \( \omega \) instead of \( \infty \), so that \( U_\alpha(H) \in \omega+1\).
Ulm’s theorem states that two countable reduced abelian \(p\)-groups
are isomorphic if and only if their Ulm lengths and Ulm invariants are the same. Consequently, two countable abelian \(p\)-groups are isomorphic if and only if they have same Ulm length, same Ulm invariants, and their divisible parts have same rank. (For a brief account of Ulm's theorem we refer the reader to \cite[Section~4]{CalTho}.)

Let \(\mathcal{L}'=\{+\}\), regarded as a relational language. For \(x\in X_{\mathcal{L}'} = 2^{\omega\times\omega\times \omega}\), denote by \(H_x\) the \(\mathcal{L}'\)-structure encoded by \(x\). More precisely, define \(H_x\) as the \(\mathcal{L}'\)-structure \((\omega, +_x)\) where \(k+_xm = n \iff x(k,m,n)\).

Fix a recursive bijection \(\langle \cdot , \cdot \rangle \colon \omega^2 \to \omega\).
Let \( \mathsf{LO} \) be the set of codes for countable (strict) linear orders, i.e., \(\mathsf{LO} \subseteq 2^\omega\) consists
 of all \(x \in 2^\omega\) such that 
 \[
 {<_x} = \{(m,n) : x(\langle m,n\rangle) = 1\}
 \] 
 is a linear order.
Let also \(\mathsf{WO} \subseteq \mathsf{LO}\) be the set of codes of well-orders of \(\omega\). If \(x \in \mathsf{WO}\), let \(|x|\) be the
order type of \(<_x\). Recall that \( \mathsf{WO} \) is a \( \Pi^1_1 \)-set which is complete for the class of coanalytic sets. (Throughout the paper we distinguish between boldface and lightface projective hierarchy.)

\begin{definition} \label{def:f}
Let \(f\colon X_{\mathcal{L}'} \times 2^\omega\to \omega + 1\) be the partial function defined as follows:
\[
f(x,y) =\begin{cases}
U_{|y|}(H_x) &\text{if \(H_x\) is an abelian \(p\)-group,}\\ &\text{ \(y\in \mathsf{WO}\) and \(|y|<l(H_x)\)};\\
\text{undefined} &\text{otherwise}.
\end{cases}
\]
\end{definition}

Next lemma states some usuful facts about the definability of Ulm classification.

\begin{lemma}
\label{lem : formulas}
\begin{enumerate-(a)}
\item \label{lem : formulas-a}
For every \(\alpha < \omega_1\) there is an \(\mathcal{L}'_{\omega_1\omega}\)-formula \(\upphi_\alpha(v)\) such that
for any abelian \(p\)-group \(H\) and any element \(a\in H\)
\[
H \models \upphi_\alpha[a] \quad\iff\quad
\alpha<l(H)\text{ and } a\in p^\alpha H.
\]
\item  \label{lem : formulas-b}
For every \(\alpha < \omega_1 \) and \(n\in \omega + 1\) there is an \(\mathcal{L}'_{\omega_1\omega}\)-sentence \(\uppsi_{\alpha,n}\) such that
for any abelian \(p\)-group \(H\) 
\[
H \models \uppsi_{\alpha,n} \quad\iff\quad
\alpha<l(H)\text{ and }  U_\alpha(H) =n.
\]
\item \label{lem : formulas-c}
For all \(x\in X_{\mathcal{L}'}\), if \(H_x\) is a \(p\)-group then
\(l(H_x)\leq\omega_1^x\).
\item \label{lem : formulas-d}
The partial function \(f\) from Definition~\ref{def:f} is \(\Pi^1_1\)-recursive.
\end{enumerate-(a)}
\end{lemma}

\begin{proof}
The proof of~\ref{lem : formulas-a} and~\ref{lem : formulas-b} is implicit in \cite[Lemmas~2.1 and 2.2]{BarEkl}. 
Part~\ref{lem : formulas-c} is the relativized version of \cite[Theorem~8.17]{AshKni}.
For part~\ref{lem : formulas-d}, by~\cite[Theorem 3D.2]{Mos} it is enough to show that the graph of \( f \) is a \(\Pi^1_1 \)-subset of \( X_{\mathcal L'} \times 2^\omega \times (\omega+1)  \), where \(\omega+1\) is identified in the obvious way with \( \omega \). This follows from the fact that \( \mathsf{WO} \) is \( \Pi^1_1 \), 
that the sequence \( ( p^\beta H_x \mid \beta \leq |y|+1 ) \) can be computed recursively-in-\( (x,y) \), and that the predicate \( n = U_{|y|}(H_x) \) is arithmetic in the parameters \( p^{|y|} H_x \) and \( p^{|y|+1} H_x \), and hence in \( (x , y) \).
\end{proof}

Given a countable ordinal \(\beta<\omega_1\) and two \( \mathcal{L} \)-structures \( H_1 \) and \( H_2 \), we write \( H_1 \equiv^\beta H_2 \) if and only if \( H_1 \) and \( H_2 \) satisfy the same \( \mathcal{L}_{\omega_1\omega} \)-sentences of quantifier rank at most \(\beta\). 

\begin{lemma}{\cite[Theorem~3.1]{BarEkl}}
\label{lem : BarEkl}
Let \(H_1\) and \(H_2\) be two abelian \(p\)-groups, viewed as \( \mathcal{L}' \)-structures, and \(\beta < \min \{l(H_1), l(H_2) \} \) be an ordinal satisfying \(\omega\beta=\beta\). Suppose that \(U_\alpha(H_1)=U_\alpha(H_2)\) for all \(\alpha<\beta\). Then \(H_1\equiv^\beta H_2\).
\end{lemma}

\subsection{The equivalence relation \( \E \)} \label{subsec:maindefinition}

Let \(\mathcal{L} = \{+,a_0, a_1, a_2 \dotsc\}\) be the expansion of \( \mathcal L' = \{ + \} \) obtained by adding countably many new constant symbols \(a_0, a_1, a_2 \dotsc\). 
For \(x\in X_\mathcal{L}\) let \(\mathcal{M}_x\) denote the \(\mathcal{L}\)-structure encoded by \(x\). (We continue to use the letter \(H\) to denote a group structure and the letter \(\mathcal{M}\) to denote a structure in the expanded language \(\mathcal{L}\).)

Let \(a_0, a_1, a_2, \dotsc\) name the elements of the infinite rank quasi-cyclic \(p\)-group \(\mathcal{D} = \mathbb{Z}(p^\infty)^{(\omega)}\), and construe the atomic diagram \( \mathrm{Diag}(\mathcal{D}) \) of \( \mathcal{D}\)
as an \(\mathcal{L}\)-theory. 
Let \(T'\) be the theory of abelian \(p\)-groups, and let 
\[
T = T' \cup \mathrm{Diag}(\mathcal{D}).
\]

For any model \(\mathcal{M}\) of \(T\), we denote by \(H(\mathcal{M})\) the \(\mathcal{L}'\)-reduct of \(\mathcal{M}\); by choice of \( T \), \( H(\mathcal{M}) \) is always an abelian \( p \)-group whose divisible part has infinite rank.

\begin{definition} \label{def:EdiBecker}
Let \(X=\Mod(T)\) be topologized as a subspace of \(X_\mathcal{L}\). Becker's equivalence relation \( \E \) on \( X \) is defined by setting 
\( x \mathrel{\E} y \iff H(\mathcal{M}_x) \cong H(\mathcal{M}_y) \).
\end{definition}

We now prove some facts that will be used later on.
Given \( \mathcal{M} \models T\),
let \( A(\mathcal{M}) \) be the \( \mathcal L \)-substructure of \( \mathcal{M} \) whose domain consists of the interpretations of the constants \( a_i \). Notice that \( A(\mathcal M) \models T \), and that \( H(A(\mathcal{M})) \cong \mathcal D \).

\begin{definition} \label{def : H(M)}
Let \(H\) be a countable group whose divisible part has infinite rank, and let \(n\in \omega+1\). 
Then \(\mathcal{M}(H,n)\) denotes the unique (up to isomorphism) countable model \(\mathcal{M}\) of \( T \) such that:
\begin{enumerate-(a)}
\item 
\( H(\mathcal M) = H \);
\item 
the divisible subgroup \( D(H(\mathcal{M})) \) is the direct sum of \( A(\mathcal{M})\) and a divisible \(p\)-group of rank \(n\).
\end{enumerate-(a)}
\end{definition}

In particular, if \(\mathcal{M}(H,m) \cong \mathcal{M}(H,n)\), then \(m=n\). Furthermore,
if \( \mathcal M \models T \), then there is a unique \( n \in \omega+1 \) such that \( \mathcal M \cong \mathcal M(H,n) \) for \( H = H(\mathcal M) \). We call \(n\) the \textbf{non-Ulm-invariant} of \(\mathcal{M}\).

\begin{lemma} \label{lem : Borel h}
\begin{enumerate-(a)}
\item \label{item : Borel h(a)}
\(\Mod(T)\) is a \( \cong \)-invariant \(G_\delta\)-subset of \(X_\mathcal{L}\).
\item \label{item : Borel h(b)}
There is a Borel map \(\theta \colon \Mod(T) \to \Mod(T)\) such that
\begin{enumerate-(1)}
\item 
\(\mathcal{M}_x\) and \(\mathcal{M}_{\theta(x)}\) have same Ulm length and same Ulm invariants, and
\item 
\( \mathcal M_{\theta(x)} \cong \mathcal M(H,\omega) \) for \( H = H(\mathcal M_x) \) (or, equivalently, \( H = H(\mathcal M_{\theta(x)}) \)).
\end{enumerate-(1)}
\item \label{item : Borel h(d)}
Let \(x,y\in\Mod(T)\) and \(\beta = \omega \beta \leq \min \{ l(H(\mathcal{M}_x)), l(H(\mathcal{M}_y)) \} \) be 
such that
\(U_\alpha (H(\mathcal{M}_x)) = U_\alpha (H(\mathcal{M}_y))\)  for all \(\alpha<\beta\). Then
\(\mathcal{M}_x\equiv^\beta \mathcal{M}_y \).
\end{enumerate-(a)}
\end{lemma}

\begin{proof}
\ref{item : Borel h(a)}
By counting quantifiers.

\ref{item : Borel h(b)}
Given \( x \in \Mod(T) \), it is easy to recover in a Borel way a code for the direct sum of \( \mathcal M_x \) and \( \mathbb{Z}(p^\infty)^{(\omega)} \).

\ref{item : Borel h(d)} 
By~\cite[Propositions 7.17 and 7.47]{Vaa}, it is enough to show that Player II has a winning strategy in the 
dynamic Ehrenfeuch-Fra\"iss\'e game \( \mathrm{EFD}_\beta(\mathcal{M}_x,\mathcal{M}_y) \).
There are abelian \(p\)-groups \(H_x \), \(H_y \) such that
\[
\mathcal{M}_x = H_x \oplus A(\mathcal M_x)\qquad \text{ and }\qquad
\mathcal{M}_y = H_y \oplus A(\mathcal M_y).
\]
The structures \(A(\mathcal M_x)\) and \( A(\mathcal M_y) \) are divisible and isomorphic via the map \( f \) sending the interpretation of a constant symbol \( a_i \) in \( \mathcal M_x \) to the interpretation of the same symbol in \( \mathcal M_y \). Therefore, 
\(H_x\) and \(H_y\) have the same Ulm length and the same Ulm invariants as \( H(\mathcal{M}_x)\) and \(H(\mathcal{M}_y)\), respectively.
By Lemma~\ref{lem : BarEkl} we get that \(H_x\equiv^\beta H_y\) as \( \mathcal L'\)-structures, hence Player II has a winning strategy \( \sigma \) in \( \mathrm{EFD}_\beta(H_x,H_y) \). On the other hand, the isomorphism \( f \) yields an obvious winning strategy \( \tau \) for player II in \( \mathrm{EFD}_\beta(A(\mathcal M_x), A(\mathcal M_y)) \). Using \( \sigma \) and \( \tau \) coordinatewise, Player II easily obtains a winning strategy in the game \( \mathrm{EFD}_\beta(\mathcal{M}_x,\mathcal{M}_y) \) over the direct products \( H_x \oplus A(\mathcal M_x) \) and \( H_y \oplus A(\mathcal M_y) \).
\end{proof}

Before introducing useful notation, we recall that Zippin's theorem~\cite[Chapter~11, Theorem~1.9]{Fuc15} implies that
for any \(\beta<\omega_1\), there is a countable reduced abelian \(p\)-group \(R^\beta\) such that \(l(R^\beta)= \beta\), and all its Ulm invariants are infinite.

\begin{definition}
For \(\beta<\omega_1\) and \( n \in \omega + 1 \), let \(\mathcal{M}^{(\beta,n)}\) denote the unique (up to isomorphism) \( \mathcal{L} \)-structure \( \mathcal{M} \) such that:
\begin{enumerate-(a)}
\item 
\( \mathcal{M} \models T \);
\item 
\(l(H(\mathcal{M})) = \beta\); 
\item 
\(U_\alpha (H(\mathcal{M}))= \omega \) for all \(\alpha<\beta\);
\item 
the non-Ulm-invariant of \(\mathcal{M}\) is \(n\).
\end{enumerate-(a)}
\end{definition}

For \(\beta<\omega_1\) and \(n\in\omega+1\), we let
\begin{align*}
\widetilde{\mathcal{M}}^{(\beta,n)} & =\{x\in\Mod(T) : \mathcal{M}_x\cong \mathcal{M}^{(\beta,n)}\} \\
\widetilde H^\beta & = \bigcup\{\widetilde{\mathcal{M}}^{(\beta,n)}:n \in \omega+1\} \\
\widetilde S & = \bigcup\{\widetilde{\mathcal{M}}^{(\beta,\omega)}:\beta<\omega_1\}
\end{align*}

\begin{lemma} \label{lem : S tilde}
The set \(\widetilde S\) is \(\boldsymbol{\Sigma}^1_1\).
\end{lemma}

\begin{proof}
The set of all \( x \in \Mod(T) \) such that the non-Ulm-invariant of \( \mathcal{M} \) is \( \omega \) is easily seen to be analytic, therefore we have only to show that the set
\[  
P = \{ x \in \Mod(T) : \forall \alpha < l(H(\mathcal{M}_x)) \, [U_\alpha(H(\mathcal{M}_x)) = \omega] \}
\]
is analytic too. The set \( \{ (y,x) \in 2^\omega\times 2^\omega  \mid y \in \Delta^1_1(x) \} \) is a proper \( \Pi^1_1 \)-set (\cite[Exercises 4D.14 and 4D.16]{Mos}). Thus there is a computable function \( g \colon 2^\omega\times 2^\omega \to \mathsf{LO} \) such that \( y \in \Delta^1_1(x) \iff g(y,x) \in \mathsf{WO} \), and for each \( x \in 2^\omega \) we have 
\[ 
\sup \{ |g(y,x)| : y \in \Delta^1_1(x) \} = \omega_1^x . 
\]
Let \( h \colon \mathsf{LO} \times \omega \to \mathsf{LO} \) be a \( \Delta^1_1 \)-map sending \( (y,k) \) to a code for \( <_y \restriction \{ m \in \omega \mid m <_y k \} \).
By Lemma~\ref{lem : formulas}\ref{lem : formulas-c}, we have
\[
x \notin P \iff x \notin \Mod(T) \vee \exists y \in \Delta^1_1(x) \, \exists n , k \in \omega \, [f(x,h(g(y,x),k)) = n],
\]
where \( f \) is as in Definition~\ref{def:f}. 
By Lemma~\ref{lem : formulas}\ref{lem : formulas-d}, the graph of \( f \) is \( \Pi^1_1 \). Then, by~\cite[Theorem 4D.3]{Mos}, the set of \(x\in \Mod(T)\) such that 
\[\exists y \in \Delta^1_1(x) \, \exists n , k \in \omega \, [f(x,h(g(y,x),k)) = n]\]
is \(\Pi^1_1.\)
It follows that the complement of \( P \) is coanalytic, and hence \( P \) is analytic.
\end{proof}

\subsection{Proof of Theorem~\ref{thm : main}}
\label{sec:Becker}

Recall that \( X = \Mod(T) \subseteq X_{\mathcal{L}} \), and that the map \( H \) associates to every \( \mathcal{M} \models T\) its \( \mathcal{L}'\)-reduct \( H(\mathcal{M}) \).
Notice that \( H \) can be construed as a Borel map \( X \to X_{\mathcal{L}'} \), hence it witnesses that \( \E \) is Borel reducible to the isomorphism relation on countable abelian \( p \)-groups. This readily implies that \( \E \) is an analytic equivalence relation all of whose equivalence classes are Borel, and that \( \E \) is classifiable by Ulm invariants. 

Notice also that by the observation after Definition~\ref{def : H(M)}, each \( E \)-equivalence class is the union of a countably infinite set of isomorphism classes of (codes for) \( \mathcal{L} \)-structures whose \( \mathcal{L'} \)-reducts are abelian \( p \)-groups with the same Ulm length and the same Ulm invariants but arbitrary non-Ulm-invariant. More in detail, for any \(\E\)-equivalence class \(C\subseteq X\) and \(n\in \omega+1\), let \( C_n \) be the collection of all \( x \in C \) such that the non-Ulm-invariant of \(\mathcal{M}_x\) is \(n\). Then \( \{ C_n : n \in \omega+1 \} \) is a partition of \( C \), and each \( C_n \) is a \(\cong \)-equivalence class. \emph{In particular, each \( \widetilde H^\beta \) constitutes an \( \E \)-equivalence class.}

To prove that \( \E \) is idealistic, note that
the Borel map \(\theta \colon X \to X\) from Lemma~\ref{lem : Borel h}\ref{item : Borel h(b)} selects an \(\cong_\mathcal{L}\)-equivalence class within every \(\E\)-class, hence we can apply Proposition~\ref{prop:select}.

Finally, we show that \( \E \) does not classwise Borel embed into any orbit equivalence relation.
This is obtained by combining Lemmas~\ref{lem:steel} and~\ref{lemma:contradiction1} below. 
The following result first appeared as Lemma~3.2 in~\cite{Bec}, where it is attributed (essentially) to Steel.

\begin{lemma} \label{lem:steel}
Assume \(\boldsymbol{\Sigma}^1_1\)-determinacy. Then for club many \(\beta< \omega_1\), the set  \(\widetilde H^\beta\) is not \(\boldsymbol{\Pi}^0_{\beta+1}\).
\end{lemma}

\begin{proof}
Recall that \(X_\mathcal{L}\) is recursively homeomorphic to the Cantor space \(2^\omega \); for the purposes of this proof, it is convenient to actually identify \(X_\mathcal{L}\) with \(2^\omega\).
For \(x\in X_\mathcal{L}\), let
\[
\omega_1^{[x]} = \inf\{\omega_1^y: \mathcal{M}_y \cong \mathcal{M}_x\}.
\]

Consider the following game \(\mathcal{G}({\widetilde{S}})\). Players I and II each play elements of \(2=\{0,1\}\), alternating moves in the usual way:

\begin{center}
\begin{tabular}{c|cccccccccc}
Player I & \(w_0\) && \(w_1\)&\(\dotsm\)&& \(w_{n}\) &&\(\dotsm\)&\\
\hline
Player II & & \(x_0\) &&&\(\dotsm\)&&\(x_{n}\)&&\(\dotsm\)
\end{tabular}
\end{center}
    
After infinitely many moves, Player I has played \(w = (w_n)_{n \in \omega} \in 2^\omega\), and Player II has played \(x = (x_n)_{n \in \omega} \in 2^\omega\). Player II wins the round of \(\mathcal{G}(\widetilde S)\) if
    
\[
x\in \widetilde S \quad\text{ and }\quad \omega_1^{[x]} \geq \omega_1^w.
\]
Set \( B = \{(x,w): \omega_1^{[x]} \geq \omega_1^w\} \).

\begin{claim} \label{claim1}
The set \( B \) is \( \Sigma^1_1 \) and the game \(\mathcal{G}(\widetilde S)\) is determined.
\end{claim}

\begin{proof}[Proof of the Claim]
Let \(j_\mathcal{L}\colon S_\infty\curvearrowright X_\mathcal{L}\) be the logic action, which is continuous and induces \( \cong \) as its associated orbit equivalence relation. By \cite[Theorem~3.2]{Sam94}, for any \(x\in X_\mathcal{L}\) the set \(\{g\in S_\infty:\omega_1^{j_\mathcal{L}(g,x)}=\omega_1^{[x]}\}\) is comeager. We claim that
\begin{equation} \label{eq:nonmeager}
\omega_1^{[x]} \geq \omega_1^w \iff \{g\in G: \omega_1^{j_\mathcal{L}(g,x)}\geq \omega_1^w\}\text{ is not meager}.
\end{equation}
The forward direction is obvious. For the backward direction, notice that if \( \omega_1^{[x]} < \omega_1^w \), then \( \{g\in G: \omega_1^{j_\mathcal{L}(g,x)}\geq \omega_1^w\} \) is contained in \(\{g\in S_\infty:\omega_1^{j_\mathcal{L}(g,x)}>\omega_1^{[x]}\}\).

Since the logic action \( j_\mathcal{L} \) is actually recursive, the set \( \{ (g,x,w) \in S_\infty \times (2^\omega)^2 \mid \omega_1^{j_\mathcal{L}(g,x)}\geq \omega_1^w \} \) is a recursive preimage of \( \{ (y,z) \in (2^\omega)^2 \mid \omega_1^y \leq \omega_1^z \}\); since the latter is \( \Sigma^1_1 \) (\cite[Lemma A.1.1]{Gao}) and  \( \Sigma^1_1\) is closed under category quantifiers (see e.g.\ \cite{Kec73} and \cite[Exercise 4F.19]{Mos}), the set \( B \) is \( \Sigma^1_1 \) by the equivalence~\eqref{eq:nonmeager}. 

The set \(\widetilde S\) is analytic by Lemma~\ref{lem : S tilde},
and hence so is the payoff set for Player II. 
Therefore we are done by \(\boldsymbol{\Sigma}^1_1\)-determinacy.
\end{proof}

\begin{claim} \label{claim2}
Player I does not have a winning strategy in \(\mathcal{G}(\widetilde S)\).
\end{claim}

\begin{proof}[Proof of the Claim]
Towards a contradiction, assume that \( \sigma \) is a winning strategy for Player I in \(\mathcal{G}(\widetilde S)\). Let \(a \in 2^\omega\) be such that \(\widetilde S\) is \(\Sigma^1_1(a)\). 
Therefore the set
\[
V= \{
x\in2^\omega : x\in \widetilde S \text{ and } \omega_1^{\langle \sigma, a\rangle} \leq \omega_1^{[x]}
\}
\]
is \(\Sigma^1_1(\langle\sigma,a\rangle )\) because it is the intersection of \( \widetilde S \) with the horizontal section of \( B \) determined by the point \( \langle \sigma, a \rangle \).
By Lemma~\ref{lem : formulas}\ref{lem : formulas-c}, if \( x \in \widetilde{M}^{(\beta,\omega)}\) and \( z \in X_{\mathcal{L}'} \) is such that \( H_z = H(\mathcal{M}_x) \), then 
\[
\omega_1^x \geq \omega_1^z \geq l(H_z) = l(H(\mathcal{M}_x)) =\beta.
\]
Therefore, \(\sup \{ \omega_1^{[x]} : x\in\widetilde S\} = \omega_1\), so that
\(V\neq \emptyset\). 
By the Gandy Basis Theorem\footnote{The Gandy basis theorem~\cite{Gan60} states that for every \( y \in 2^\omega \), each nonempty \(\Sigma^1_1(y)\) set
contains some element \(x_0\) with \(\omega_1^{\langle y, x_0\rangle}\leq \omega_1^{y}\).}, there is \(x_0 \in V\) such that \(\omega_1^{\langle \sigma, a,x_0\rangle} \leq \omega_1^{\langle \sigma, a\rangle}\). Consider a round of the game in which Player II plays such a \(x_0 \in V \subseteq \widetilde S\) and Player I, following the strategy \(\sigma\), plays some \( w \in 2^\omega\). Clearly, \(w\) is recursive-in-\(\langle\sigma,x_0\rangle\), hence
\[
\omega_1^w \leq
\omega_1^{\langle\sigma,x_0 \rangle}
\leq \omega_1^{\langle \sigma, a,x_0\rangle}\leq
\omega_1^{\langle\sigma,a \rangle} \leq
\omega_1^{[x_0]}.\]
Thus, Player II wins this round of the game, contradicting the fact that \(\sigma\) is winning for Player I.
\end{proof}

Let \(v \in 2^\omega\). Recall that a countable ordinal \(\alpha\) is \markdef{\(v\)-admissible} if there is a \(y \in  2^\omega\) such that \(\alpha = \omega_1^{\langle v,y\rangle}\).

\begin{claim} \label{claim3}
If Player II has a winning strategy \( \sigma \) in \(\mathcal{G}(\widetilde S)\), then 
\[
\{ \beta< \omega_1 : \beta \text{ is \(\sigma\)-admissible and } \widetilde H^\beta \text{ is not } \boldsymbol{\Pi}^0_{\beta+1} \}
\]
contains a club.
\end{claim}

\begin{proof}[Proof of the Claim] 
Let \(F\colon 2^\omega\to 2^\omega\) be the function \(F(w)= w\ast \sigma\).
Let
\[
W = \left\{\beta<\omega_1 : \beta \text{ is \(\sigma\)-admissible and } \omega_1^{[x]}<\beta \text{ for all \(x\in \bigcup\nolimits_{\alpha < \beta} \widetilde M^{(\alpha,\omega )}\)} \right\}.
\]
By~\cite[Theorem 0.11]{Bec94}, the set \( \{ \beta < \omega_1 : \beta \text{ is \( \sigma \)-admissible} \} \) contains a club.
Using standard arguments, one sees that also 
\[
\left\{ \beta < \omega_1 : \omega_1^{[x]}<\beta \text{ for all \(x\in \bigcup\nolimits_{\alpha < \beta} \widetilde M^{(\alpha,\omega )}\)} \right\}
\] 
contains a club.
Therefore \(W\) contains a club. 

\emph{In what follows, we heavily use some notation and standard arguments related to Steel's forcing; the reader unfamiliar with such technique might consider reading Appendix~\ref{appendix} before continuing with the rest of this proof.}
Fix any \( \beta \in W \) and denote by \(\mathbb{P}_\beta\) the Steel's forcing poset of level \(\beta\). Let \(g\) be \( \mathbb{P}_\beta \)-generic over \( L(\sigma) \), and let \(w \in2^{\omega}\) be the (code for the) tree \( t(g) \) obtained from \( g \) by removing its tags. (The existence of generics for \(\mathbb{P}_\beta\) is discussed in Proposition~\ref{prop:existenceofgenerics}.)
We are going to show that \(F(w)\in\widetilde M^{(\beta,\omega)}\).
Since \( \beta \) is \( \sigma \)-admissible, we have \(\omega_1^w = \omega_1^{\langle w,\sigma \rangle}= \beta\). (Fact~\ref{fact:omega1}.) 
Since \(\sigma\) is a winning strategy for Player II, we have \(F(w) \in \widetilde S\) and \(\omega_1^{[F(w)]}\geq \beta\). 
On the other hand, \(F(w)\) is recursive-in-\(\langle w,\sigma\rangle\), hence \(\omega_1^{[F(w)]} \leq \omega_1^{F(w)} \leq \omega_1^{\langle w, \sigma \rangle }= \beta\). Hence, \( \omega_1^{[F(w)]} = \beta \). 
By definition of \(\widetilde S\) and \( F(w) \in \widetilde S \), there is a unique \( \alpha < \omega_1 \) such that \( F(w) \in \widetilde M^{(\alpha,\omega )}\).  Let \( x \in [F(w)]_{\cong} = \widetilde M^{(\alpha,\omega)} \) be such that \( \omega_1^x = \omega_1^{[F(w)]} \). It follows from Lemma~\ref{lem : formulas}\ref{lem : formulas-c} that 
\[
\alpha = l(H(\mathcal M_x)) \leq \omega_1^x = \omega_1^{[F(w)]} = \beta.
\]
But the case \( \alpha < \beta \) is excluded by the definition of \( W \),
hence \( \alpha = \beta \) and \(F(w) \in \widetilde M^{(\beta, \omega)}\).

We now prove that \( \widetilde H^\beta \notin \boldsymbol{\Pi}^0_{\beta+1} \) for every \( \beta \in W \), which concludes the proof of the claim. 
Let \(\beta\in W\) be arbitrary, and let \(\gamma\) be the least element of \(W\) greater than \(\beta\). If \( \widetilde H^\beta \in \boldsymbol{\Pi}^0_{\beta+1} \), then also \(Q = F^{-1}(\widetilde H^{\beta}) \in \boldsymbol{\Pi}^0_{\beta+1} \) because \( F \) is continuous.
Since \(\widetilde M^{(\beta,\omega)}\subseteq \widetilde H^\beta\) and \(\widetilde M^{(\gamma,\omega) }\cap \widetilde H^\beta=\emptyset\), the previous paragraph shows that, in the jargon of Steel's forcing, \( Q \) separates \( \mathbb P_\beta\)-generics from \( \mathbb P_\gamma \)-generics over \( L(\sigma) \) (see the paragraph after Fact~\ref{fact:omega1} for the precise definition). But this contradicts Proposition~\ref{prop : complexity separating set}, hence \( \widetilde H^\beta \notin \boldsymbol{\Pi}^0_{\beta+1} \).
\end{proof}
Putting Claims~\ref{claim1}, \ref{claim2}, and~\ref{claim3} together, the proof of Lemma~\ref{lem:steel} is complete.
\end{proof}

The following lemma is instead modeled after~\cite[Lemma~3.1]{Bec}. Together with Lemma~\ref{lem:steel}, it shows that under \(\boldsymbol{\Sigma}^1_1\)-determinacy the equivalence relation \( \E \) does not classwise Borel embed into any orbit equivalence relation, and thus it concludes the proof of Theorem~\ref{thm : main}. We will use the following simple observation.

\begin{lemma}[Folklore] \label{lem:opentransformed}
Let \( E_a \) be the orbit equivalence relation induced by a continuous action \( a \colon G \times Z \to Z \) of a Polish group \( G \) on a Polish space \( Z \). For every open set \( U \subseteq X \), its Vaught's transform 
\begin{align*} 
U^{\triangle} & = \{ z \in Z : \exists^* g \in G \, (a(g,z) \in U) \} \\
& = \{ z \in Z : \{ g \in G : a(g,z) \in U \} \text{ is not meager} \}
\end{align*}
coincides with its \( E_a \)-saturation \( [U]_{E_A} \).
\end{lemma}

\begin{proof}
For the nontrivial inclusion, let \(  z \in Z \) and \( g \in G \) be such that \( a(g,z) \in U \): we want to show that \( z \in U^\triangle \). Since the action is continuous, there is an open neighborhood \( V \subseteq G \) of \( g \) such that \( a(g',z) \in U \) for every \( g' \in V \). Since \( G \) is Baire, \( V \) is nonmeager, and thus it witnesses \( z \in U^\triangle \).
\end{proof}

\begin{lemma} \label{lemma:contradiction1}
Suppose that \( \E \embeds_{cB} E_a\) for some orbit equivalence relation \( E_a \) induced by a Borel action \(a\colon G\times Z \to Z\) of a Polish group \(G\) on a standard Borel space \(Z\). Then for club many \(\beta<\omega_1\), the set \(\widetilde H^\beta\) is \(\boldsymbol{\Pi}^0_{\beta+1}\).
\end{lemma}

\begin{proof}
Recall that the domain \( X = \Mod(T) \) of \( \E \) is a Polish space when equipped with the topology \( \tau \) induced by \( X_{\mathcal{L}}\) (Lemma~\ref{lem : Borel h}\ref{item : Borel h(a)}).
Suppose that \(f\colon X \to Z\) witnesses that \( \E \sqsubseteq_{cB} F\) with classwise Borel inverse\footnote{In the sense of Definition~\ref{def:classwise}.} \(g\colon [f(X)]_{E_a} \to X\). 
Since \( [f(X)]_{E_a} \) is Borel and \( E_a \)-invariant, without loss of generality we may assume that \([f(X)]_{E_a} = Z\). By~\cite[Theorem~5.2.1]{BecKec}, the standard Borel space \(Z\) can be given a Polish topology \(t\) so that the action \(a\colon G\times Z\to Z\) is continuous.

Let \( \omega \leq \gamma < \omega_1 \) be such that both \(f \) and \( g\) are \(\boldsymbol{\Sigma}^0_\gamma\)-measurable. Then for every \(\delta \geq \gamma\omega\) and every \(A\in \boldsymbol{\Sigma}^0_\delta(Z,t)\), we have \(f^{-1}(A)\in \boldsymbol{\Sigma}^0_\delta(X,\tau)\). 
Indeed, an easy induction on \( \beta < \omega_1 \) shows that \(f^{-1}(B)\in \boldsymbol{\Sigma}^0_{\gamma+\beta}(X,\tau)\)
for all \(B \in \boldsymbol{\Sigma}^0_{1+\beta}(Z,t)\). In particular, \( f^{-1}(B) \in \boldsymbol{\Sigma}^0_{\gamma \cdot (n+1)}(X,\tau) \) for every \(B \in \boldsymbol{\Sigma}^0_{\gamma \cdot n} (Z,t)\) because \( \gamma \geq \omega \), and thus 
\(f^{-1}(A)\in \boldsymbol{\Sigma}^0_{\gamma \omega}(X,\tau)\) for all \(A\in \boldsymbol{\Sigma}^0_{\gamma\omega}(Z,t)\).
This settles the case \( \delta = \gamma \omega \). The general case \( \delta \geq \gamma \omega \) can then be proved by an easy induction. A similar argument shows that \( g^{-1}(A) \in \boldsymbol{\Sigma}^0_\delta(Z,t) \) for every \( \delta \geq \gamma \omega \) and \( A \in \boldsymbol{\Sigma}^0_\delta(X,\tau) \).

Fix \(r\in 2^\omega\) such that \(G\) is recursively-in-\(r\) presented and \(a\colon G\times Z\to Z\), \(f \colon X \to Z\), and \(g \colon Z \to X\) are all \(\Delta^1_1(r)\)-functions.

Let \( W \) be the collection of all 
\( \beta < \omega_1 \) such that 
\begin{itemizenew}
\item 
\( \beta > \omega_1^r \) and \( \beta \geq \gamma \omega\);
\item 
\( \omega \beta = \beta \) (in particular, \( \beta \) is limit); 
\item 
for every \( \alpha < \beta \) and \( n \in \omega+1\), the formulas \( \upphi_\alpha(v)\) and \(\uppsi_{\alpha, n}\) from Lemma~\ref{lem : formulas} define%
\footnote{Each \( \mathcal{L}'_{\omega_1 \omega} \)-formula \( \upvarphi(v_1, \dotsc, v_k) \) defines the set
\[ \left \{ (x, n_1, \dotsc, n_k) \in X_{\mathcal{L}'} \times \omega^k  \mid \mathcal{M}_x \models \upvarphi[n_1, \dotsc, n_k ]  \right\} .\]}
Borel subsets of rank less than \(\beta\) (with respect to \( \tau \)).
\end{itemizenew}

We are going to show that \(\widetilde{H}^\beta \in \boldsymbol{\Pi}^0_{\beta+1}(X,\tau) \) for every \(\beta\in W\): since \( W \) is a  club,
this concludes the proof of the lemma.

Fix any \(\beta \in W\). Let
\[
Y=\{
x\in X : l(H(\mathcal{M}_x)) \geq \beta\text{ and } \forall  \alpha<\beta\,  (U_\alpha(H(\mathcal{M}_x)) = \omega) \}.
\]
By Lemma~\ref{lem : Borel h}\ref{item : Borel h(d)},
for any two \(x,y\in Y\) we have \(\mathcal{M}_x\equiv^\beta \mathcal{M}_y\).
Moreover, \(Y\) is \( \E\)-invariant (hence also \(\cong \)-invariant), and \( Y \in \boldsymbol{\Pi}^0_\beta(X,\tau) \) 
by the last condition in the definition of \(W\). 
Notice also that the \( \E \)-equivalence class \( \widetilde H^\beta \) is contained in \( Y \).

Our argument splits in two parts. Following~\cite[Claim on p.\ 22]{Bec}, we first observe that 
every \( \E \)-equivalence class \(C\subseteq Y\) with \(C\neq \widetilde H^\beta\) is not in \(\boldsymbol{\Pi}^0_{\beta+1}(X,\tau)\). Then we show that there must be an \(\E\)-equivalence class contained in \( Y \) which belongs \(\boldsymbol{\Pi}^0_{\beta+1}(X,\tau)\): necessarily, such class is \( \widetilde H^\beta \) and we are done.

\begin{claim} \label{claim4}
Let \(C\subseteq Y\) be an \(E\)-equivalence class other than \(\widetilde H^\beta\). Then \(C \notin \boldsymbol{\Pi}^0_{\beta+1}(X,\tau)\).
\end{claim}

\begin{proof}[Proof of the Claim]
Recall that \( \mathcal{L} \) is obtained by adding to \( \mathcal{L}' = \{ + \} \) infinitely many constants \( a_n \). By Lemma~\ref{lem : formulas}\ref{lem : formulas-a}--\ref{lem : formulas-b}, for all \(x\in X\) we have that \(x\in \widetilde{\mathcal{M}}^{(\beta,0)}\) if and only if 
it belongs to the intersection of the following three sets:
\begin{align*}
& \bigcap_{\alpha < \beta} \left\{ x \in X \mid \mathcal{M}_x \models \exists v \, (\upphi_\alpha(v) \wedge \neg \upphi_{\alpha+1}(v)) \right \} ; \\
& \bigcap_{i \in \omega} \, \bigcup_{j,k \in \omega} \, \bigcup_{\alpha < \beta }\, \left\{ x \in X \mid \mathcal{M}_x \models i = j+k \wedge \upphi_\alpha[j] \wedge \neg \upphi_{\alpha+1}[j] \wedge \bigvee_{n \in \omega} k = a_n \right\} ; \\
& \bigcap_{\alpha < \beta } \left\{ x \in X \mid \mathcal{M}_x \models \uppsi_{\alpha,\omega} \right \}.
\end{align*}
By the third condition in the definition of \( W \) and the fact that \( \beta \) is limit, the first and last sets are in \( \boldsymbol{\Pi}^0_\beta(X,\tau) \), while the middle one belongs to \( \boldsymbol{\Pi}^0_{\beta+1}(X,\tau) \). Therefore \(\widetilde{\mathcal{M}}^{(\beta,0)} \), which is a subset of \( \widetilde H^\beta \), belongs to \(\boldsymbol{\Pi}^0_{\beta+1}(X,\tau)\).

Towards a contradiction, suppose that \(C\subseteq Y\) is an \( \E\)-equivalence class such that \( C \neq \widetilde H^\beta \) and \( C \in \boldsymbol{\Pi}^0_{\beta+1}(X,\tau) \).
In particular, \( C \) is disjoint from \( \widetilde{\mathcal{M}}^{(\beta,0)} \).
By Sami's Transfer Lemma~\cite[Section 4.3]{Sam94}  there is a Polish topology \( \tau'\) on \(X\) such that:
\begin{itemizenew}
\item 
the logic action \(j_\mathcal{L}\) is still continuous if \( X \) is equipped with \( \tau' \);
\item 
\( \tau'\) has a countable basis consisting of sets in \( \bigcup_{1 \leq \alpha < \beta} \boldsymbol{\Pi}^0_\alpha(X,\tau)\);
\item
both \( C\) and \(\widetilde {\mathcal{M}}^{(\beta,0)}\) are \(G_\delta\) with respect to \( \tau'\).
\end{itemizenew}

We claim that there is a nonempty \( \tau' \)-basic open set \( D_0\) such that either
\begin{align*}
 D_0 \cap  C\neq \emptyset  \quad &\text{ and } \quad D_0 \cap \widetilde{\mathcal{M}}^{(\beta,0)}=\emptyset, \qquad \text{or else}\\
D_0 \cap  C =\emptyset  \quad &\text{ and } \quad D_0 \cap \widetilde{\mathcal{M}}^{(\beta,0)}\neq\emptyset.
\end{align*}
To see this,
let \( X' \subseteq X \) be the \( \tau' \)-closure of \( C \cup \widetilde{\mathcal{M}}^{(\beta,0)} \). If all \( \tau' \)-basic open sets \( D \) with \( D \cap  X' \neq \emptyset \) meet both \(C\) and \( \widetilde{\mathcal{M}}^{(\beta,0)}\), then such sets would be dense in \(  X' \), and hence comeager because they are \( G_\delta \) with respect to \( \tau' \): since \( C \cap \widetilde{\mathcal{M}}^{(\beta,0)} = \emptyset \), this contradicts that fact that \( X' \), being Polish, is a Baire space.

Consider the Vaught's transform \( D_0^\triangle \) of \( D_0 \) with respect to the logic action \( j_{\mathcal{L}} \), which is a \( \cong \)-invariant set. On the one hand, working in the space \( (X,\tau') \) and applying Lemma~\ref{lem:opentransformed} we get that \( D_0^\triangle = [D_0]_{\cong} \) because \( D_0 \) is \( \tau' \)-open. 
Since both \( C \) and \( \widetilde{\mathcal{M}}^{(\beta,0)}\) are \( \cong \)-invariant, then \( D_0 \cap C = \emptyset \iff [D_0]_{\cong} \cap C = \emptyset \) and \( D_0 \cap \widetilde{\mathcal{M}}^{(\beta,0)} = \emptyset \iff [D_0]_{\cong} \cap \widetilde{\mathcal{M}}^{(\beta,0)} = \emptyset \), and hence in all cases there are \( x,y \in Y \) such that \( x \in D_0^\triangle \) and \( y \notin D_0^\triangle \). On the other hand, the Borel rank of \( D_0^\triangle \) with respect to the original topology \( \tau \) is some ordinal \( \alpha < \beta  \) because this is true of \( D_0 \) (see \cite[Lemma~5.1.7(g)]{BecKec}).
Therefore, \( D_0^\triangle \) can be defined by an \(\mathcal{L}_{\omega_1\omega}\)-sentence \( \uptheta \) of quantifier rank at most \(\omega\alpha\) by~\cite[Theorem 16.8]{Kec}. 
 (Indeed, following the proof of~\cite[Proposition 16.9]{Kec} one sees that when the Borel rank of the \( \cong \)-invariant set at hand is increased by \( 1 \), then the \( \mathcal{L}_{\omega_1 \omega} \)-formula defining it needs only finitely many additional quantifiers with respect to the ones used in the previous levels, and no new quantifier is needed when dealing with limit levels. An easy induction on \( \alpha \) gives then the desired result.) 
 As \( x \in D_0^\triangle \) and \( y \notin D_0^\triangle \), \( \mathcal{M}_x \models \uptheta \) while \( \mathcal{M}_y \not\models \uptheta \). Since \( \omega \alpha < \omega \beta = \beta \),
 it follows that \( \mathcal{M}_x \not\equiv^\beta \mathcal{M}_y \), a contradiction with the fact that \( x,y \in Y \) entails \( \mathcal{M}_x \equiv^\beta \mathcal{M}_y \). 
 \end{proof}

To conclude the proof, we just need
to show that there is at least an \( \E \)-equivalence class \(C\subseteq Y\), which by Claim~\ref{claim4} is necessarily \( \widetilde H^\beta \),  that belongs to \(\boldsymbol{\Pi}^0_{\beta+1}(X,\tau)\).
Since \(Y \in \boldsymbol{\Pi}^0_\beta(X,\tau)\), then \(Y' = g^{-1}(Y) \in \boldsymbol{\Pi}^0_\beta(Z,t) \subseteq \boldsymbol{\Pi}^0_{\beta+1}(Z,t) \) because \(\beta \geq \gamma\omega\), 
and \( Y' \) is clearly \( E_a \)-invariant. Recall also that \( \beta \in W \) is a countable limit ordinal. Notice that \( E_a \restriction Y' \) is classifiable by Ulm invariants because \( E_a \leq_B \E \) via \( g\), hence \(E_0 \not\sqsubseteq_c E_a\restriction Y'\)
by~\cite[Theorem~3.4.4]{BecKec}. 
By~\cite[Corollary~5.1.10]{BecKec},
there is an \(E_a\)-equivalence class \(C'\subseteq Y'\) such that \(C' \in \boldsymbol{\Pi}^0_{\beta+1}(Z,t)\), and since \(\beta \geq \gamma\omega\), the \(\E\)-equivalence class \(C = f^{-1}(C') \subseteq Y \) belongs to \(\boldsymbol{\Pi}^0_{\beta+1}(X,\tau)\).
\end{proof}

\section{Some questions} \label{sec:questions}

Despite being discussed since the early nineties, the notion of idealistic is one of the most mysterious and intriguing notions in invariant descriptive set theory. 
Besides Question~\ref{question:basic}, which is still open as stated, and in particular in the Borel case, there are many natural and fundamental questions that still need to be answered. In the list below we collect some of the questions that are tightly related to the content of this manuscript, but we also include other more general problems about idealistic equivalence relations.

\begin{question}
\label{q:no determinacy}
Can the \( \boldsymbol{\Sigma}^1_1\)-determinacy assumption be removed from Theorem~\ref{thm : main}?
\end{question}

The equivalence relation \(\E\) is in the G\"odel's constructible universe \(L\), and the definition of \( \mathcal{I} \) is \(\Pi^1_3\). Therefore, under analytic determinacy Theorem~\ref{thm : main} is true in \(L\) and this might suggest that no assumption beyond \( \mathsf{ZF} \) is necessary.

\begin{question}
\label{q:Borel in I}
Does \( \mathcal{I} \) contain Borel equivalence relations?
\end{question}
We remark that the equivalence relation \(\E\) is not Borel, otherwise we would have \( E_0 \leq \E \) by the Borel Glimm-Effros dichotomy; indeed, Question~\ref{q:Borel in I} is open even under \(\boldsymbol{\Sigma}^1_1\)-determinacy.

As already observed in Section~\ref{sec:Becker},
\(\E\) is Borel reducible to an orbit equivalence relation, and precisely to \(\cong_\mathcal{L}\). However, the following problem remains open.

\begin{question}
Is 
\(\E\) Borel bireducible with an orbit equivalence relation? 
\end{question}

Notice that, by construction, \(\E\) is Borel bireducible with the restriction of an orbit equivalence relation to an invariant analytic set.

A major problem in the area
is to understand whether we can have any sort of dichotomy theorem concerning idealistic equivalence relations.
For example, since Hjorth's \( \EH \) is Borel reducible to a countable Borel (hence orbit) equivalence relation, the following conjecture remains unaffected by the results of this paper.

\begin{conjecture} \label{conj3intro} 
Let \( E \) be a Borel equivalence relation. Then either \( E_1 \leq_B E \) or \( E \) is Borel reducible to an orbit (or even just idealistic) equivalence relation.
\end{conjecture}

Conjecture~\ref{conj3intro} is one of the most important and longstanding structural problems in the area, and it deserves full attention from the descriptive set-theoretic community. It appears in~\cite[Question 10.9]{Hjo}
and is discussed in~\cite{Kec2009} (see the diagram on p.\ 101 and the explanations following it).
Solving in the positive another problem from~\cite{KecLou97},
Solecki has verified Conjecture~\ref{conj3intro} in the special case of equivalence relations of the form \( 2^\omega / \mathcal{I} \) for \( \mathcal{I} \) an ideal on \( \omega \) which is Borel as a subset of \( 2^\omega  \) (see~\cite{Sol}).
However, the general case remains open.

To the best of our knowledge, the following questions are also open.

\begin{questions} \label{q:final}
\begin{enumerate-(1)}
\item
Is there a complete idealistic equivalence relation?
\item
Are idealistic equivalence relations closed under the Friedman-Stanley jump operator?
\item \label{q:weak}
Is every idealistic (Borel) equivalence relation Borel reducible to a orbit equivalence relation? Can such orbit equivalence relation be taken to be Borel is so is the idealistic equivalence relation we start with?
\item \label{q:omega1}
Is there any idealistic equivalence relation with exactly \(\omega_1\)-many equivalence classes? 
\item
Does the fact that \( E \) is idealistic imply that every \( E \)-equivalence class is Borel?
\item\label{q:E_1 not reducible}
Is it true that \(E_1\not \leq_B E\) for any idealistic (not necessarily Borel) equivalence relation \(E\)?
\end{enumerate-(1)}
\end{questions}

Question~\ref{q:final}\ref{q:weak} is a weakening of Question~\ref{question:basic} and is clearly strictly related to Proposition~\ref{prop:intro2}. Indeed, if the answer to Question~\ref{q:final}\ref{q:weak} is positive, then also Question~\ref{question:basic} could be answered positively for Borel equivalence relations, provided that one can extend H.\ Friedman's result~\cite{HFri} from actions of \( S_\infty \) to actions of arbitrary Polish groups (see~\cite[Question 10.8]{Hjo}).
A negative answer to Question~\ref{q:final}\ref{q:omega1} would imply the longstanding topological Vaught conjecture, so it is very challenging.
Question~\ref{q:final}\ref{q:E_1 not reducible} has been answered by Kechris and Louveau in the special cases when \( E \) is either idealistic Borel, or an orbit equivalence relation (\cite[Theorem~4.1, and Theorem~4.2]{KecLou97}).

\appendix
\section{Steel's forcing} \label{appendix}

We give an almost self-contained presentation of Steel's forcing with tagged trees, that is, trees equipped with a function assigning to each of their nodes a tag in \( \omega_1 \cup \{ \infty \} \). The main goal is to prove Proposition~\ref{prop : complexity separating set}, that is used in the proof of Theorem~\ref{thm : main}.

Given an ordinal \(\alpha\), we denote by \(\mathbb{P}_\alpha\)  the \textbf{Steel's forcing poset} of level \(\alpha\).
A condition in \(\mathbb{P}_\alpha\) is a tagged tree \(\langle T, h\rangle\) such that:
\begin{enumerate-(1)}
\item
\(T\) is a finite tree on \(\omega\), i.e., a finite subset of \(\omega^{<\omega}\) such that if \(\tau \in T\) and \(\sigma \subseteq \tau \) then \(\sigma \in T\).
\item 
\(h\colon T \to\omega\alpha \cup \{\infty\}\) is a function such that
\begin{enumerate-(i)}
\item 
\(h(\emptyset) = \infty\);
\item
\(h(\tau)<h(\sigma)\) for all \(\sigma\subset\tau\) in \(T\), where \( < \) is the usual strict order on ordinals extended to \( \infty \) by stipulating that \(\beta<\infty\) for all \(\beta \in \omega \alpha \).
\end{enumerate-(i)}
\end{enumerate-(1)}

If \(p=\langle S, h_S \rangle\) and \(q=\langle T, h_T \rangle\) are two conditions in \(\mathbb{P}_\alpha\), we define
\[p\leq_\alpha q \quad \iff\quad S \supseteq T 
\text{ and }h_S\restriction T =h_T.\]

Steel's forcing only deals with a particular kind of infinitary formulas, called here tree-formulas, rather than arbitrary first-order formulas in the language of set theory. 
\begin{definition}
\textbf{Tree-formulas} are defined recursively as follows:
\begin{enumerate-(i)} 
\item 
\textbf{Atomic} tree-formulas are those of the form \( \sigma \in X \), for some \( \sigma \in \omega^{< \omega} \).
\item 
If \(\upvarphi\) is a tree-formula, then \(\neg\upvarphi\) is a tree-formula too.
\item 
If \((\upvarphi_n)_{n \in \omega}\) is a countable sequence of tree-formulas, then \(\bigwedge_{n \in \omega}\upvarphi_n\) is a tree-formula.
\end{enumerate-(i)}
\end{definition}

Next, we introduce a specific notion of rank for tree-formulas.
\begin{definition}
The \textbf{rank} \( \rnk(\upvarphi) \) of a tree-formula \( \upvarphi \) is defined by recursion on the complexity of \( \varphi \) as follows:
\begin{enumerate-(i)}
\item 
If \( \upvarphi \) is atomic, then \(\rnk(\upvarphi) = 1\).
\item 
If \( \upvarphi \) is of the form \( \neg \uppsi \), then \(\rnk (\upvarphi) = \rnk(\uppsi)+1\).

\item 
If \( \upvarphi \) is of the form \( \bigwedge_{n \in \omega} \upvarphi_n \), then \(\rnk (\upvarphi) = \sup\{\rnk(\upvarphi_n)+1 : n \in \omega\} \).

\end{enumerate-(i)}
\end{definition}

Still working by recursion on its complexity, to each tree-formula \( \varphi \) we can also canonically associate a set \( Q_\upvarphi \subseteq 2^{(\omega^{< \omega})}\).

\begin{definition} \label{def:Q_phi}
\begin{enumerate-(i)}
\item 
If \( \upvarphi \) is an atomic formula of the form \( \sigma \in X \), then \( Q_\upvarphi = \{ T \in 2^{(\omega^{< \omega})} \mid \sigma \in T \} \).
\item 
If \( \upvarphi \) is of the form \( \neg \uppsi \), then \( Q_\upvarphi = 2^{(\omega^{< \omega})} \setminus Q_\uppsi \).
\item 
If \( \upvarphi\) is of the form \( \bigwedge_{n \in \omega} \upvarphi_n \), then \( Q_\upvarphi = \bigcap_{n \in \omega} Q_{\upvarphi_n} \).
\end{enumerate-(i)}
\end{definition}

Given a set \( Q \subseteq 2^{(\omega^{< \omega})}\), we say that that a tree-formula \( \upvarphi \) \textbf{represents} \( Q \)  if \( Q = Q_\upvarphi \).

\begin{lemma} \label{lem:Q_phi}
For every tree formula \( \upvarphi \), the set \( Q_\upvarphi \) is Borel.
Conversely, if \( 1 \leq \xi < \omega_1 \) and  \( Q \subseteq 2^{(\omega^{< \omega})} \) is \( \boldsymbol{\Pi}^0_\xi \) (respectively: \( \boldsymbol{\Sigma}^0_\xi \)), then \( Q \) is represented by a tree-formula \( \upvarphi \) with rank \( 3+2 \xi\) (respectively: \( 3 + 2 \xi + 1 \)).  
\end{lemma}

In particular, if \( \xi = n \in \omega\), then \( 3 + 2\xi = 2n+3 \), while if \( \xi = \lambda + n \) for some limit \( \lambda \) and \( n \in \omega \), then \( 3 + 2 \xi = \lambda + 2n \).

\begin{proof}
If \( \upvarphi \) is atomic, then \( Q_\upvarphi \) is in the canonical subbase for the product topology on \( 2^{(\omega^{< \omega})}\). 
Thus, by induction one sees that \( Q_\upvarphi \) is always in the smallest \( \sigma \)-algebra generated by such sets, and it is thus Borel.

For the second part, we argue by induction on \( 1 \leq \xi < \omega_1 \). 
Since sets in \( \boldsymbol{\Sigma}^0_\xi \) are complements of sets in \( \boldsymbol{\Pi}^0_\xi \), we just need to use one more negation to represent them. Thus it is enough to consider the case of \( \boldsymbol{\Pi}^0_\xi \).
For the basic case \( \xi = 1 \), first observe that the sets in the canonical subbase for the product topology are either of the form \( \{ T \in 2^{(\omega^{<\omega})} \mid T(\sigma) = 1 \} \) or \( \{ T \in 2^{(\omega^{<\omega})} \mid T(\sigma) = 0 \} \), for \( \sigma \in \omega^{< \omega} \). The former is represented by \( \sigma \in X \), which has rank \( 1 \), the latter is represented by \( \neg (\sigma \in X ) \), which has rank \( 2 \). Thus the basic open sets are represented by a conjuction of tree-formulas as above, which has thus rank \( 3 \). This implies that closed sets, being intersections of complements of basic open sets, are represented by tree-formulas of rank \( 5 \).

Suppose now that \( \xi = \zeta+1\) is a successor ordinal. By inductive hypothesis, sets in \( \boldsymbol{\Sigma}^0_\zeta \) are represented by tree-formulas of rank \( 3 + 2\zeta + 1 \). It follows that every set in \( \boldsymbol{\Pi}^0_{\xi} \), being an intersection of \( \boldsymbol{\Sigma}^0_\zeta \)-sets, is represented by a tree-formula of rank 
\[ 
3 + 2\zeta + 1 + 1 = 3 + 2(\zeta+1) = 3 + 2 \xi.
\]

Finally, suppose that \( \xi \) is limit, and let \( (\xi_n)_{n \in \omega} \) be a strictly increasing sequence of ordinals cofinal in \( \xi \). Pick any \( A \in \boldsymbol{\Pi}^0_\xi \), and let \( A_n \in \boldsymbol{\Sigma}^0_{\xi_n} \) be such that \( A = \bigcap_{n \in \omega } A_n \). Then by inductive hypothesis each \( A_n \) is represented by a tree-formula \( \upvarphi_n \) with \( \rnk(\upvarphi_n) = 3 + 2 \xi_n + 1 \). Therefore \( A \) is represented by the tree-formula \( \bigwedge_{n \in \omega} \upvarphi_n \), which has rank
\[  
\sup \{ 3 + 2 \xi_n + 1 + 1 : n \in \omega \} = \sup \{ 3 + 2 (\xi_n+1) : n \in \omega \} = 3 + 2 \xi. \qedhere
\]
\end{proof}

\begin{definition}
Given an ordinal \( \alpha \), we define the Steel's forcing relation \( \Vdash_\alpha \) as follows, where \( p = (T,h) \) is any condition in \( \mathbb{P}_\alpha \) and \( \upvarphi \) is a tree-formula:
\begin{enumerate-(i)}
\item if \( \upvarphi \) is an atomic tree-formula of the form \( \sigma \in X \), then \(p\Vdash_\alpha \upvarphi \) if and only if either \(\sigma \in T\) or \(\sigma\) is an immediate successor of some \(\tau \in T\) with \(h(\tau)\neq 0\);
\item
\(p\Vdash_\alpha \neg \upvarphi\) if and only if  for all \(q\leq_\alpha p\) we have \(q\not\Vdash_\alpha\upvarphi\);
\item 
\(p\Vdash_\alpha \bigwedge_{n \in \omega} \upvarphi_n\) if and only if \(p\Vdash_\alpha \upvarphi_n\) for all \(n \in \omega\).
\end{enumerate-(i)}
\end{definition}

A useful feature of Steel's forcing is the so called Retagging Lemma. By definition, if \(\alpha\leq\beta \) then \(\mathbb{P}_\alpha \subseteq \mathbb{P}_\beta\). In this case, there is a function 
\[
\mathbb{P}_\beta\to \mathbb{P}_\alpha, \quad p\mapsto p^\alpha, 
\]
called \textbf{retagging function}, which respects the partial orderings and is the identity function on \(\mathbb{P}_\alpha\):
for \(p=\langle T,h\rangle\in \mathbb{P}_\beta\), the retagging of \(p\) in \(\mathbb{P}_\alpha\) is the condition \(p^\alpha=\langle T,h'\rangle\) with 

\[
h'(\sigma)=
\begin{cases}
    h(\sigma) & h(\sigma)<\omega\alpha\\
    \infty & \text{otherwise}.
\end{cases}
\]
It is easy to check that: 
\begin{itemizenew}
\item 
\(p^\alpha = p\) for all \(p\in \mathbb{P}_\alpha\);
\item 
for any \(\alpha\leq\beta \), the retagging function from \(\mathbb{P}_\beta \) to \( \mathbb{P}_\alpha \) is monotone, that is, \(p\leq_\beta q \implies p^\alpha\leq_\alpha q^\alpha\) for every \( p,q \in \mathbb{P}_\beta \).
\end{itemizenew}

\begin{theorem}[Retagging Lemma] \label{thm:retagging}
If \(\beta <\gamma < \omega_1\) and \(\upvarphi\) is a tree-formula with \(\rnk(\upvarphi) < \beta\), then for every \( p \in \mathbb{P}_\gamma \) we have
\(
p \Vdash_\gamma \upvarphi \iff p^\beta \Vdash_\beta \upvarphi.
\)
\end{theorem}

Any generic filter \(g\subseteq \mathbb{P}_\alpha\) can be canonically identified with a tagged tree by taking unions. The tree \(t(g)\) is obtained from such a tagged tree by removing the tags, that is, 
\[
t(g) = \bigcup\{T : (T,h) \in g \text{ for some } h \}.
\]
Formally, \( t(g) \) is a subset of \( \omega^{< \omega} \), and thus \( t(g) \in 2^{(\omega^{<\omega})}\). However, the latter can be canonically identified with \( 2^\omega \), hence when needed we can conveniently regard \( t(g) \) as an element of the Cantor space, so that it makes sense to e.g.\ consider \( \omega_1^{t(g)} \) and alike.

For any tree-formula \( \upvarphi \), there is an interesting link between the forcing relation \( p \Vdash_\beta \upvarphi\) and the set \( Q_\upvarphi \) represented by \( \upvarphi \) introduced in Definition~\ref{def:Q_phi}.

\begin{proposition} \label{prop:Q_phi}
Let \( \alpha < \omega_1 \), \( \upvarphi \) be a tree-formula, and \( g \) be \( \mathbb{P}_\alpha \)-generic.
Then
\[
t(g) \in Q_\upvarphi \qquad \text{if and only if} \qquad p \Vdash_\alpha \upvarphi \text{ for some } p \in g.
\]
\end{proposition}

The above proposition can be proved by induction on the complexity of \( \upvarphi \), using standard arguments. For a detailed exposition, see e.g.\ \cite[Theorem 3.5]{Lut}, keeping in mind that \( t(g) \in Q_\upvarphi \) if and only if, in Lutz's notation, \( t(g) \models \upvarphi \).

For the rest of this section, we fix \( r \in 2^\omega \) and analyze what happens when using various Steel's forcings \( \mathbb{P}_\alpha \) over \( L(r) \).

\begin{proposition}[Folklore] \label{prop:existenceofgenerics}
Assume \(\boldsymbol{\Sigma}^1_1\)-determinacy, and fix any \( \alpha < \omega_1 \). Then for every \( p \in \mathbb{P}_\alpha \) there is a \( \mathbb{P}_\alpha \)-generic \( g \) over \(L(r)\) such that \( p \in g \).
\end{proposition}

\begin{proof}
Work in \( L(r) \).
Since \( |\mathbb{P}_\alpha| < \omega_1^V \), there are at most \( 2^{|\mathbb{P}_\alpha|}\)-many dense subsets of \( \mathbb{P}_\alpha \), 
and since under our assumptions \( \omega_1^V \) is inaccessible in \( L(r) \)~\cite[Corollary~33.11]{Jec}, there are strictly less than \( \omega_1^V\)-many of them. Therefore, stepping back into \( V \), there are only countably many dense subsets of \( \mathbb{P}_\alpha \) which belong to \( L(r) \), and thus the required \( \mathbb{P}_\alpha \)-generic \( g \) can easily be constructed using the standard argument (see e.g.\ \cite[Lemma 14.4]{Jec}).
\end{proof}

Recall that a countable ordinal \(\alpha\) is \markdef{\(r\)-admissible}, for some \(r \in 2^\omega\),  if there is \(y \in  2^\omega\) such that \(\alpha = \omega_1^{\langle r,y\rangle}\).
We will also need the following fact, due to Steel (see~\cite[Theorems~2.9 and 2.10]{Har78}). 

\begin{fact} \label{fact:omega1}
If \(\alpha < \omega_1 \) is \(r\)-admissible and \(g\) is \(\mathbb{P}_\alpha\)-generic
over \(L(r)\), then \(\omega_1^{t(g)} = \omega_1^{\langle t(g),r\rangle} = \alpha\). 
\end{fact}

Finally, fix \(\beta<\gamma < \omega_1\), and suppose that there are both \( \mathbb{P}_\beta \)-generics and \( \mathbb{P}_\gamma \)-generics over \( L(r) \). Then we say that \(Q\subseteq 2^{(\omega^{< \omega})}\) \textbf{separates \(\mathbb{P}_\beta\)-generics from  \(\mathbb{P}_\gamma\)-generics over \( L(r) \)}  if the following conditions hold:
\begin{enumerate-(a)}
\item 
If \(g\) is \(\mathbb{P}_\beta\)-generic over \( L(r) \), then \(t(g)\in Q\);
\item 
If \(g\) is \(\mathbb{P}_\gamma\)-generic over \( L(r) \), then \(t(g)\notin Q\).
\end{enumerate-(a)}

The Retagging Lemma~\ref{thm:retagging} yields the following useful fact, in which the \(\boldsymbol{\Sigma}^1_1\)-determinacy assumption is only used to ensure the existence of \( \mathbb{P}_\beta \)-generics and \( \mathbb{P}_\gamma \)-generics over \( L(r) \). 

\begin{proposition} \label{prop : complexity separating set}
Assume \(\boldsymbol{\Sigma}^1_1\)-determinacy.
Suppose that for some \(\beta<\gamma < \omega_1 \) with \( \beta \) limit, the set \(Q \subseteq 2^{(\omega^{< \omega})}\) separates \(\mathbb{P}_\beta\)-generics from  \(\mathbb{P}_\gamma\)-generics over \( L(r) \). 
Then \(Q\) is not \(\boldsymbol{\Pi}^0_{\beta+1}\).
\end{proposition}

\begin{proof} 
Towards a contradiction, suppose that \(Q\) is \(\boldsymbol{\Pi}^0_{\beta+1}\). 
Write \( 2^{(\omega^{< \omega} )} \setminus Q \) as \( \bigcup_{i \in \omega} R_i \) for some \( R_i \in \boldsymbol{\Pi}^0_\beta \), and for each \( i \in \omega \)
write \( R_i \) as \( \bigcap_{k \in \omega} S^{i}_k \), where \( S^{i}_k \in \boldsymbol{\Sigma}^0_{\alpha_{i,k}} \) for some \( \alpha_{i,k} < \beta \). 
For each \( i,k \in \omega \), let \( \upvarphi_{i,k} \) be a tree-formula with \( \rnk(\upvarphi_{i,k}) < \beta \) representing \( S^{i}_k \)  (here we use Lemma~\ref{lem:Q_phi} and the fact that \( \beta \) is limit), so that by Definition~\ref{def:Q_phi} the tree-formula \( \bigwedge_{k \in \omega} \upvarphi_{i,k}\), abbreviated by \( \upvarphi_i \), represents \( R_i \), and  
\( \bigwedge_{i \in \omega} \neg \upvarphi_i \), which will also be denoted by \( \upvarphi \), represents \( Q \).

Let \(g\) be any \(\mathbb{P}_\beta\)-generic over \( L(r) \). Since \( t(g) \in Q = Q_\upvarphi \) by the separation assumption, using Proposition~\ref{prop:Q_phi} we get that there is \( p_1 \in g \subseteq \mathbb{P}_\beta \) such that \( p_1\Vdash_\beta \upvarphi \),
so that 
\( p_1 \Vdash_\beta \neg \upvarphi_i \) for every \( i \in \omega \), 
and hence 
\begin{equation}\tag{\(\ddagger\)} \label{eq:Steelseparation}
p \not\Vdash_\beta \upvarphi_i \quad \text{for every \( p \leq_\beta p_1 \) and \( i \in \omega \)}.
\end{equation}

On the other hand, since \( \mathbb{P}_\beta \subseteq \mathbb{P}_\gamma \) we have that there is a \( \mathbb{P}_\gamma \)-generic \( g' \) over \( L(r) \) such that \( p_1 \in g' \) (Lemma~\ref{prop:existenceofgenerics}). Since \( t(g') \notin Q = Q_\upvarphi \) by the separation assumption again, \( p \not\Vdash_\gamma \upvarphi \) for all \( p \in g' \). 
Specializing this to \( p = p_1  \) and recalling that \( \upvarphi \) is \( \bigwedge_{i \in \omega} \neg \upvarphi_{i} \), we get that there is \( \bar\imath \in \omega \) such that \( p_1 \not\Vdash_\gamma \neg \upvarphi_{\bar\imath} \), which in turn implies that \( p_2 \Vdash_\gamma \upvarphi_{\bar \imath} \) for some \( p_2 \leq_\gamma p_1\). Since \( \upvarphi_{\bar\imath} \) is \( \bigwedge_{k \in \omega} \upvarphi_{\bar\imath,k}\), this means that \( p_2 \Vdash_\gamma \upvarphi_{\bar\imath,k}\) for every \( k \in \omega \).
By the Retagging Lemma~\ref{thm:retagging}, we have \( p_2^\beta \Vdash_\beta \upvarphi_{\bar\imath,k} \) for all \( k \in \omega \), and hence \( p_2^\beta \Vdash_\beta \upvarphi_{\bar\imath}\). But since \( p^\beta_2 \leq_\beta p^\beta_1 = p_1 \) (because \( p_1 \in \mathbb{P}_\beta \)), this contradicts~\eqref{eq:Steelseparation}.
\end{proof}


\begin{thebibliography}{10}

\bibitem{AdaKec}
Scot Adams and Alexander~S. Kechris.
\newblock Linear algebraic groups and countable {B}orel equivalence relations.
\newblock {\em J. Amer. Math. Soc.}, 13(4):909--943, 2000.

\bibitem{AshKni}
Christopher~J. Ash and Julia Knight.
\newblock {\em Computable structures and the hyperarithmetical hierarchy}, volume 144 of {\em Studies in Logic and the Foundations of Mathematics}.
\newblock North-Holland Publishing Co., Amsterdam, 2000.

\bibitem{BarEkl}
Jon Barwise and Paul Eklof.
\newblock Infinitary properties of abelian torsion groups.
\newblock {\em Ann. Math. Logic}, 2(1):25--68, 1970/1971.

\bibitem{Bec94}
Howard Becker.
\newblock The topological {V}aught's conjecture and minimal counterexamples.
\newblock {\em The Journal of Symbolic Logic}, 59(3):757--784, 1994.

\bibitem{Bec}
Howard Becker.
\newblock Idealistic equivalence relations.
\newblock Unpublished notes, 2001.

\bibitem{BecKec}
Howard Becker and Alexander~S. Kechris.
\newblock {\em The descriptive set theory of {P}olish group actions}, volume 232 of {\em London Mathematical Society Lecture Note Series}.
\newblock Cambridge University Press, Cambridge, 1996.

\bibitem{CalTho}
Filippo Calderoni and Simon Thomas.
\newblock The bi-embeddability relation for countable abelian groups.
\newblock {\em Trans. Amer. Math. Soc.}, 371(3):2237--2254, 2019.

\bibitem{DouJacKec}
Randall Dougherty, Stephen Jackson, and Alexander~S. Kechris.
\newblock The structure of hyperfinite borel equivalence relations.
\newblock {\em Transactions of the American Mathematical Society}, 341(1):193--225, 1994.

\bibitem{FriSta}
Harvey Friedman and Lee Stanley.
\newblock A {B}orel reducibility theory for classes of countable structures.
\newblock {\em J. Symbolic Logic}, 54(3):894--914, 1989.

\bibitem{HFri}
Harvey~M. Friedman.
\newblock Borel and {B}aire reducibility.
\newblock {\em Fund. Math.}, 164(1):61--69, 2000.

\bibitem{FriMot}
Sy-David Friedman and Luca Motto~Ros.
\newblock Analytic equivalence relations and bi-embeddability.
\newblock {\em J. Symbolic Logic}, 76(1):243--266, 2011.

\bibitem{Fuc15}
L{\'a}szl{\'o} Fuchs.
\newblock {\em Abelian groups}.
\newblock Springer Monographs in Mathematics. Springer, Cham, 2015.

\bibitem{Gan60}
Robin~O. Gandy.
\newblock On a problem of {K}leene’s.
\newblock {\em Bulletin of the American Mathematical Society}, 66:501--502, 1960.

\bibitem{Gao01}
Su~Gao.
\newblock Some dichotomy theorems for isomorphism relations of countable models.
\newblock {\em J. Symbolic Logic}, 66(2):902--922, 2001.

\bibitem{Gao}
Su~Gao.
\newblock {\em Invariant descriptive set theory}, volume 293 of {\em Pure and Applied Mathematics (Boca Raton)}.
\newblock CRC Press, Boca Raton, FL, 2009.

\bibitem{Har78}
Leo Harrington.
\newblock Analytic determinacy and \( 0^\# \).
\newblock {\em The Journal of Symbolic Logic}, 43(4):685--693, 1978.

\bibitem{Hjo}
Greg Hjorth.
\newblock {\em Classification and orbit equivalence relations}, volume~75 of {\em Mathematical Surveys and Monographs}.
\newblock American Mathematical Society, Providence, RI, 2000.

\bibitem{Hjo05}
Greg Hjorth.
\newblock Bi-{B}orel reducibility of essentially countable {B}orel equivalence relations.
\newblock {\em The Journal of Symbolic Logic}, 70(3):979--992, 2005.

\bibitem{HjoKec}
Greg Hjorth and Alexander~S. Kechris.
\newblock New dichotomies for {B}orel equivalence relations.
\newblock {\em Bull. Symbolic Logic}, 3(3):329--346, 1997.

\bibitem{HjoKec2}
Greg Hjorth and Alexander~S. Kechris.
\newblock Recent developments in the theory of {B}orel reducibility.
\newblock {\em Fund. Math.}, 170(1-2):21--52, 2001.
\newblock Dedicated to the memory of Jerzy \L o\'s.

\bibitem{Jec}
Thomas Jech.
\newblock {\em Set theory}.
\newblock Springer Monographs in Mathematics. Springer-Verlag, Berlin, 2003.
\newblock The third millennium edition, revised and expanded.

\bibitem{Kec73}
Alexander~S. Kechris.
\newblock Measure and category in effective descriptive set theory.
\newblock {\em Ann. Math. Logic}, 5:337--384, 1972/73.

\bibitem{Kec92}
Alexander~S. Kechris.
\newblock Countable sections for locally compact group actions.
\newblock {\em Ergodic Theory Dynam. Systems}, 12(2):283--295, 1992.

\bibitem{Kec94}
Alexander~S. Kechris.
\newblock Countable sections for locally compact group actions. ii.
\newblock {\em Proc. Amer. Math. Soc.}, 120:241--247, 1994.

\bibitem{Kec}
Alexander~S. Kechris.
\newblock {\em Classical descriptive set theory}, volume 156 of {\em Graduate Texts in Mathematics}.
\newblock Springer-Verlag, New York, 1995.

\bibitem{Kec2009}
Alexander~S. Kechris.
\newblock Set theory and dynamical systems.
\newblock In C.~Glymour et~al., editor, {\em Logic, Methodology and Philosophy of Science, Proc. of the Thirteenth International Congress}, pages 97--107. College Publ., London, 2009.

\bibitem{KecLou97}
Alexander~S. Kechris and Alain Louveau.
\newblock The classification of hypersmooth {B}orel equivalence relations.
\newblock {\em J. Amer. Math. Soc.}, 10(1):215--242, 1997.

\bibitem{KecMac}
Alexander~S. Kechris and Henry~L. Macdonald.
\newblock Borel equivalence relations and cardinal algebras.
\newblock {\em Fund. Math.}, 235(2):183--198, 2016.

\bibitem{Lut}
Patrick Lutz.
\newblock How to use {S}teel forcing.
\newblock Unpublished notes.

\bibitem{Mos}
Yiannis~N. Moschovakis.
\newblock {\em Descriptive set theory}, volume 100 of {\em Studies in Logic and the Foundations of Mathematics}.
\newblock North-Holland Publishing Co., Amsterdam-New York, 1980.

\bibitem{Mot12}
Luca Motto~Ros.
\newblock On the complexity of the relations of isomorphism and bi-embeddability.
\newblock {\em Proc. Amer. Math. Soc.}, 140(1):309--323, 2012.

\bibitem{Sam94}
Ramez~L. Sami.
\newblock Polish group actions and the {V}aught conjecture.
\newblock {\em Transactions of the American Mathematical Society}, 341(1):335--353, 1994.

\bibitem{Sol}
S{\l}awomir Solecki.
\newblock Analytic ideals and their applications.
\newblock {\em Ann. Pure Appl. Logic}, 99(1-3):51--72, 1999.

\bibitem{Sol09}
S{\l}awomir Solecki.
\newblock The coset equivalence relation and topologies on subgroups.
\newblock {\em American Journal of Mathematics}, 131(3):571--605, 2009.

\bibitem{Vaa}
Jouko V\"a\"an\"anen.
\newblock {\em Models and games}, volume 132 of {\em Cambridge Studies in Advanced Mathematics}.
\newblock Cambridge University Press, Cambridge, 2011.

\end{thebibliography}
\end{document}